\documentclass{amsart}

\usepackage{amssymb}	
\usepackage{tikz}
\usepackage{tikz-cd}

\usepackage{amsmath}
 
\usepackage{mathrsfs}
\usepackage{amsthm}
\usepackage{verbatim}
\usepackage{subeqnarray}
\usepackage{graphicx}
\usepackage{cancel}
 
\usepackage{color}
 
\usepackage{float}
 
\usepackage{bm}
\usepackage{hyperref}
 
\usepackage{amsfonts}
 
\usepackage{amscd}
\usepackage{cases}
\usepackage{xcolor}
\usepackage{eucal}

\definecolor{red}{rgb}{0.8,0,0}
\definecolor{darkorange}{rgb}{1,0.4,0}
\definecolor{lightorange}{rgb}{1,0.6, 0}
\definecolor{yellow}{rgb}{1,0.8, 0}

\newtheorem{theorem}{Theorem}
\newtheorem{remark}{Remark}

\newtheorem{lemma}{Lemma}

\newcommand{\0}{\mathaccent23}

\newcommand\tr{\operatorname{tr}}
\newcommand\inc{\operatorname{inc}}
\newcommand\skw{\operatorname{skw}}
\newcommand\sskw{\operatorname{sskw}}
\newcommand\vskw{\operatorname{vskw}}
\newcommand\mskw{\operatorname{mskw}}

\newcommand\sym{\operatorname{sym}}
\newcommand\grad{\operatorname{grad}}
\newcommand\deff{\operatorname{def}}
\renewcommand\div{\operatorname{div}}
\renewcommand\ker{\mathcal{N}}

\newcommand\curl{\operatorname{curl}}
\newcommand\rot{\operatorname{rot}}

\newcommand\dev{\mathrm{dev}}

\newcommand\hess{\operatorname{hess}}

\newcommand\ran{\mathcal{R}}

\newcommand\K{\mathbb{K}}
\newcommand\M{\mathbb{M}}

\renewcommand\S{{\mathbb S}}

\newcommand\R{\mathbb{R}}

\newcommand\x{\times}

\newcommand\V{{\mathbb{V}}}

\newcommand{\Symm}{\odot}

\usepackage{chngpage}

\usepackage{lipsum}

\usepackage{booktabs}

\newcommand{\bs}{{\scriptscriptstyle \bullet}}

\usepackage{geometry}
\let\Cal\mathcal
\let\Bbb\mathbb
\let\frak\mathfrak

\begin{document}
\title{BGG sequences with weak regularity and applications}
\author{Andreas \v{C}ap \and  Kaibo Hu}
\address{Faculty of Mathematics, University of Vienna, Oskar-Morgenstern-Platz 1, 1090 Wien, Austria}
\email{Andreas.Cap@univie.ac.at}
 \address{Mathematical Institute,
University of Oxford, Andrew Wiles Building, Radcliffe Observatory Quarter, 
Oxford, OX2 6GG, UK}
\email{Kaibo.Hu@maths.ox.ac.uk}
\date{\today.}

\maketitle

\begin{abstract}
We investigate some Bernstein-Gelfand-Gelfand (BGG) complexes consisting of Sobolev spaces on bounded Lipschitz domains in $\mathbb{R}^{n}$. In particular, we compute the cohomology of the conformal deformation complex and the conformal Hessian complex in the Sobolev setting. The machinery does not require algebraic injectivity/surjectivity conditions between the input spaces, and allows multiple input complexes. As applications, we establish a conformal Korn inequality in two space dimensions with the Cauchy-Riemann operator and an additional third order  operator with a background in M\"obius geometry. We show that the linear Cosserat elasticity model is a Hodge-Laplacian problem of  a twisted de Rham complex. From this cohomological perspective, we propose potential generalizations of continuum models with microstructures.
\end{abstract}
 
\vspace{+0.5cm}

In this article, we study BGG complexes with a view towards applications in numerical analysis. The origin of these are the Bernstein-Gelfand-Gelfand (BGG) resolutions from (infinite dimensional) representation theory. These admit a dual interpretation in terms of invariant differential operators on generalized flag manifolds. A direct construction of BGG resolutions in the language of differential operators was given in  \cite{vcap2001bernstein}. Representation theory still is an important ingredient to this construction, but only finite dimensional representations are needed there. The construction actually applies not only to generalized flag manifolds but also to manifolds endowed with curved geometries modeled on these homogeneous spaces, so-called parabolic geometries, see \cite{parabolbook}. In this more general setting, the construction does not provide complexes, but it is important as a conceptual construction of differential operators intrinsic to parabolic geometries, whose existence was not established in general before.

For the purpose of the current article, the strong invariance properties of the construction are not important, we only work in the setting of appropriate bounded domains in $\mathbb R^n$. This leads to substantial simplifications compared to the original constructions which (for the examples we consider) are related either to projective differential geometry or to conformal differential geometry. What is important for us is that one obtains an explicit construction of the relevant complexes from the de Rham complex of differential forms with values in an appropriate finite dimensional vector space. Together with known results on the de Rham complex, this allows us to obtain analytical results which are important for the use of BGG complexes in numerical analysis and applied mathematics. Still, the geometric background is helpful, for example in understanding the action of Euclidean motions, compare with the example discussed in Section \ref{sec:conf-Hess}.

The applications of the BGG construction in numerical analysis were motivated by the construction of finite elements for linear elasticity (the Hellinger-Reissner principle) \cite{Arnold2006a,arnold2002mixed}. For elasticity problems, the stress tensor and load fit in the so-called elasticity complex (Riemannian deformation complex), which is a special case of a BGG complex  \cite{Arnold2006a,eastwood2000complex,eastwood1999variations}. For numerical stability (inf-sup conditions \cite{babuvska1973finite,brezzi1974existence}) and approximation properties, it is desirable that  the finite element spaces for the stress and load fit into a discrete elasticity complex. This perspective inspired the construction of the Arnold-Winther element \cite{arnold2002mixed}, which solved a decades-long problem. 

Mathematical and numerical analysis of PDE models often involves broader classes of function spaces, e.g., Sobolev spaces. Analytic issues and applications of the complexes, particularly the elasticity complex, were discussed by several authors \cite{amstutz2019incompatibility,angoshtari2015differential,ciarlet2007characterization,geymonat2009hodge,pauly2020elasticity,pauly2020divdiv,yavari2020applications}. Inspired by the BGG construction, Arnold and Hu~\cite{arnold2021complexes} systematically derived a comprehensive list of complexes and established their algebraic and analytic properties for a large class of function spaces.  
Recently, there has been a surge of interests in discretizing these BGG complexes \cite{arf2021structure,chen2020discrete,chen2020finite,chen2021finite,chen2021finite2,chen2021geometric,christiansen2020discrete,christiansen2019finite,hu2021conforming,hu2021conforming2,sander2021conforming}.   Not only these BGG complexes, but also the ``BGG diagrams'' for deriving these complexes are important. To discretize elasticity with standard ``finite element differential forms'' \cite{Arnold.D;Falk.R;Winther.R.2006a,Arnold.D;Falk.R;Winther.R.2010a}, Arnold, Falk and Winther \cite{arnold2007mixed} proposed a scheme which imposes the symmetry of tensors weakly using Lagrange multipliers. Diagram chasing on the BGG diagrams plays an important role in this work. Babu\v{s}ka and Brezzi theories require that the discrete spaces in the BGG diagram should not be arbitrary but rather satisfy some algebraic conditions that are similar to the continuous level \cite{arnold2021complexes}. If finite element spaces are perfectly matched in the BGG diagrams, the schemes with weakly imposed symmetry imply strong symmetry \cite{gopalakrishnan2012second}. Diagram chasing on the BGG diagram also leads to null-homotopy operators for the elasticity complex \cite{christiansen2020poincare} and provides a constructive tool for deriving finite elements and complexes \cite{christiansen2020discrete,christiansen2018nodal,christiansen2019finite}.

However, the simplification in the geometric and algebraic setting in \cite{arnold2021complexes} left several questions open. Firstly, examples in \cite{arnold2021complexes} include the conformal deformation complex and the conformal Hessian complex which have important applications in general relativity and continuum mechanics  (c.f., Beig and Chrusciel  \cite{beig2020linearised} on  parameterizing the Einstein constraint equations using differential complexes, where the ``conformal complex'' and the ``scalar complex'' correspond to the  conformal deformation complex  and the  conformal Hessian complex, respectively). The cohomology of these complexes was left open in \cite{arnold2021complexes}. Secondly, \cite{arnold2021complexes} shows that analytic results, e.g., various Poincar\'e or Korn inequalities, are deeply rooted in the algebraic structures.  However, the conformal Korn inequality in two space dimensions remains open as the diagram does not fulfill the assumptions in \cite{arnold2021complexes} (c.f., \cite{dain2006generalized} for applications of generalized Korn inequalities in general relativity).  
Last but not least, the algebraic conditions (injectivity/surjectivity of the connecting maps) in \cite{arnold2021complexes} raise challenges for constructing numerical methods. Only carefully designed discrete spaces can fulfill these conditions. 

In this paper, we provide a major generalization  {of} \cite{arnold2021complexes} to solve the above issues. We get to a construction that is closer to the developments in \cite{vcap2001bernstein} but takes into account the simplifications present in our situation. In particular, the generalizations of the framework include:
\begin{itemize}
\item allowing more than two rows in the diagram,
\item removing the injectivity/surjectivity conditions in the assumptions. The resulting complexes will have a more complicated form, which can be simplified to the cases {treated} in \cite{arnold2021complexes} when algebraic conditions hold. 
\end{itemize}
As a result of this generalization, we compute the cohomology of important complexes, e.g., the conformal deformation complex, the conformal Hessian complex and higher order generalizations of the Hessian complex. We establish a conformal Korn inequality in two space dimensions and generalizations of the linear Cosserat elasticity model based on its connections to the twisted de Rham complexes. The above generalizations of the algebraic framework play a critical role in these developments.

\medskip
 
\paragraph{\it Conformal Korn inequalities.}
A generalized version of the Korn inequality holds in $n$D for $n\geq 3$. In the framework of \cite{arnold2021complexes}, this inequality holds since  the conformal deformation complex has finite dimensional cohomology (thus the operators have closed range).  Nevertheless, this result does not hold in 2D \cite{dain2006generalized}. This is consistent with the fact that a 2D version of the conformal deformation BGG diagram does not satisfy the injectivity/surjectivity conditions, thus not fitting in the framework of  \cite{arnold2021complexes}.  In fact, in 2D the $\dev\sym\grad$ operator involved in the generalized Korn inequalities corresponds to the Cauchy-Riemann operator in complex analysis. Based on the generalizations in this paper, we fix the 2D conformal Korn inequality by adding a third order term, which has a geometric interpretation via M\"obius structures.    

\paragraph{\it Linear Cosserat continuum.}
In 1909, the Cosserat brothers introduced a new continuum model that incorporated a microscopic rotational degree of freedom at each point of the material \cite{cosserat1909theorie}.  
This seminal work triggered a number of generalizations in continuum mechanics, e.g., Eringen's micropolar elasticity theory \cite{eringen1999theory}.  In this paper, we observe that the linear Cosserat elasticity exactly corresponds to the Hodge-Laplacian problem of the twisted de Rham complex that leads to the elasticity complex. A similar observation holds for plate models. The Kirchhoff and a modified version of the Reissner-Mindlin plate models are the Hodge-Laplacian problems of the Hessian complex and the corresponding twisted de Rham complex \cite{arnold2021complexes}, respectively. This connection inspired the construction of numerical methods for solving the biharmonic equation \cite{feec-lecture,quenneville2015new}, although not formulated in terms of the twisted de Rham complex. 
From this perspective, the cohomology-preserving projection in \cite{arnold2021complexes} and this paper can be understood as the process of eliminating the rotational degrees of freedom from the Cosserat model to get the classical elasticity. The operators mapping the BGG complex back to the twisted de Rham complex can in turn be interpreted as describing classical elasticity as the Cosserat model with the additional rotational degree of freedom appropriately fixed. A similar conclusion holds for the Kirchhoff and the modified Reissner-Mindlin plate models.  { Moreover, the Reissner-Mindlin plate and the Kirchhoff plate can be obtained as a dimension reduction of the Cosserat elasticity and the standard elasticity, respectively. A rigorous justification of the reduction in terms of $\Gamma$-limits can be found in \cite{ciarlet1997mathematical,neff2010reissner}.} These connections are sketched in the following diagram:
\begin{equation}\label{diagram:models}
\begin{tikzcd}[row sep=large, column sep = huge]
{\rm Cosserat~elasticity} \arrow{r}{\mathrm{BGG~(elasticity)}}\arrow[d, "{\rm  dimension~reduction}"]& {\rm classical~elasticity}\arrow[d, "{\rm  dimension~reduction}"] \\
{\rm (modified) ~Reissner-Mindlin~plate}\arrow{r}{\mathrm{BGG~(hessian)}}&{\rm Kirchhoff~plate}
\end{tikzcd}
\end{equation}

This cohomological perspective of the linear Cosserat model has several consequences, e.g., well-posed Hodge-Laplacian boundary value problems and important tools for the analysis and computation of the Cosserat model. Moreover, each of the BGG complexes leads to a Hodge-Laplacian problem (or, equivalently, energy functional).  Other complexes and diagrams may thus provide generalizations of the Cosserat model to other types of microstructures. In this paper, we show some examples in this direction inspired by the Erlangen program, where different Lie symmetries correspond to different geometries \cite{klein1893comparative}.

The rest of this paper will be organized as follows. In Section 1, we establish the abstract framework. In Section 2, we show examples that fit in the general framework. In Section 3, we discuss the algebraic and geometric background of our construction and its relation to BGG sequences for parabolic geometries. In Section 4, we focus on the conformal Korn inequalities. In Section 5, we show connections between the Cosserat model and the twisted de Rham complex and discuss potential generalizations. 

 {  Some spaces and operators  used in this paper are summarized as follows.
\begin{table}[h!]
\begin{center}
\begin{tabular}{c|c}
$Z^{i, j}$ & \eqref{diag:multi-rows} \\
$S^{i, j}: Z^{i, j}\to Z^{i+1, j-1}$ & \eqref{diag:multi-rows}\\
$d^{i, j}: Z^{i, j}\to Z^{i+1, j}$ & \eqref{diag:multi-rows}\\
 $K^{i, j}: Z^{i, j}\to Z^{i, j-1}$ & \eqref{diag:multi-rows}\\
 $T^{i, j}: Z^{i, j}\to Z^{i-1, j+1}$ & \eqref{def:T}\\
 $Z^{i}$, $d^{i}$, $S^{i}$, $T^{i}$ &vectorized versions, e.g.\, \eqref{def:Zi}\\

 $d_{V}^{i}=d^{i}-S^{i}: Z^{i}\to Z^{i+1}$ &  \eqref{dV}\\
 $\Upsilon^{i}=\ker(S^{i})\cap \ran(S^{i-1})^{\perp}\subset Z^{i}$ &\eqref{def:Upsilon}\\
 $D^{i}: \Upsilon^{i}\to \Upsilon^{i+1}$ & \eqref{def:D}\\
 \begin{tabular}{@{}c@{}}$P_{\ker(S^{i, j})}: Z^{i, j}\to \ker (S^{i, j})\subset Z^{i, j}$, \\ $P_{\ran(S^{i, j})^{\perp}}: Z^{i+1, j-1}\to \ran(S^{i, j})^{\perp}\subset Z^{i+1, j-1}$\end{tabular}  & p. \pageref{def:T}\\
 $F^{i}: Z^{i}\to Z^{i}$ & \eqref{matrix-F}\\
 $G^{i}: Z^{i}\to Z^{i-1}$ & \eqref{def:H}\\
  $A^{i}: \Upsilon^{i}\to Z^{i}$, $B^{i}: Z^{i}\to \Upsilon^{i}$ & \eqref{def:A}, \eqref{def:B}
\end{tabular}
\end{center}
\end{table}}
Throughout the paper, we use $H^{q}$ to denote the Sobolev space consisting of $L^{2}$ functions  with all derivatives up to order $q$ in $L^{2}$, and  $\|\cdot\|_{q}$ for the corresponding Sobolev norm of $H^{q}$. We use $\|\cdot\|$ to denote  the $L^{2}$ norm.

\section{Abstract framework}\label{sec:framework}

\subsection{The basic setup}

Let $Z^{i, j}, ~0\leq i\leq n, 0\leq j\leq N$, be Hilbert spaces and $d^{i, j}: Z^{i, j}\to Z^{i+1, j}$  bounded linear operators such that for each $j$ we obtain a compex $(Z^{\bs,j},d^{\bs,j})$, i.e.~$d^{i+1,j}\circ d^{i,j}=0$ for all $i,j$. Let $S^{i, j}: Z^{i, j}\to Z^{i+1, j-1}$ and $K^{i,j}:Z^{i,j}\to Z^{i, j-1}$ be linear operators  with properties specified below. We collect these data into the following diagram called the {\it BGG diagram}:
 \begin{equation}\label{diag:multi-rows}
\begin{tikzcd}
0 \arrow{r} &Z^{0, 0}  \arrow{r}{d^{0, 0}} &Z^{1, 0} \arrow{r}{d^{1, 0}} &\cdots \arrow{r}{d^{n-1, 0}} & Z^{n, 0} \arrow{r}{} & 0\\
0 \arrow{r}&Z^{0, 1}\arrow{r}{d^{0, 1}} \arrow[u, "{K^{0,1}}"]\arrow[ur, "S^{0, 1}"]&Z^{1, 1}  \arrow{r}{d^{1, 1}} \arrow[ur, "S^{1, 1}"]\arrow[u, "{K^{1,1}}"]&\cdots \arrow{r}{d^{n-1, 1}}\arrow[ur, "S^{n-1, 1}"] \arrow[u, " "]&\arrow[u, "{K^{n,1}} "] Z^{n, 1}\arrow{r}{} & 0\\
 & \cdots \arrow[u, " {K^{0,2}}"]\arrow[ur, " S^{0, 2}"]&  \cdots\arrow[ur, "S^{1, 2} "]\arrow[u, "{K^{1,2}} "]&\cdots \arrow[ur, "S^{n-1, 2} "] \arrow[u, " "]&\cdots\arrow[u, "{K^{n,2}} "]  & \\
 0 \arrow{r}&Z^{0, N}\arrow{r}{d^{0, N}} \arrow[u, "{K^{0,N}} "]\arrow[ur, "S^{0, N}"]&Z^{1, N}  \arrow{r}{d^{1, N}} \arrow[ur, "S^{1, N}"]\arrow[u, " {K^{1,N}}"]&\cdots \arrow{r}{d^{n-1, N}}\arrow[ur, "S^{n-1, N}"] \arrow[u, " "]&\arrow[u, "{K^{N,N}} "] Z^{n, N}\arrow{r}{} & 0.
 \end{tikzcd}
\end{equation}
This is a generalization of the setup in \cite{arnold2021complexes} where only two rows were used.

 The abstract results below are established in two steps that impose slightly different assumptions. However, we collect the basic relations between the operators here. We start with two basic assumptions: 
\begin{equation}\label{SKKS}
S^{i, j-1}K^{i, j}=K^{i+1, j-1}S^{i, j},
\end{equation}
\begin{equation}\label{DKKD}
S^{i, j}=d^{i, j-1}K^{i, j}-K^{i+1, j}d^{i, j}.
\end{equation}
Using the fact that the $d^{i,j}$ form a complex for each $j$, \eqref{DKKD} readily implies that 
\begin{equation}\label{DSSD}
S^{i+1, j}d^{i, j}=-d^{i+1, j-1}S^{i, j}, \quad \forall i, j \geq 0.
\end{equation}
Using \eqref{DKKD}, \eqref{DSSD} and \eqref{SKKS}, we get $S^{i+1,j-1}\circ S^{i,j}=d^{i+1,j-2}S^{i,j-1}K^{i,j}+K^{i+1,j-1}S^{i+1,j}d^{i,j}$. Inserting for both $S$ operators {in the right hand side} from \eqref{DKKD} we conclude that  
\begin{equation}\label{SS}
S^{i+1, j-1}\circ S^{i, j}=0, \quad \forall i, j \geq 0. 
\end{equation}

In applications, each row of \eqref{diag:multi-rows} will be an input complex which is well understood, most often a de Rham complex with values in a finite-dimensional vector space. The $S$ operators that connect the basic input complexes are of algebraic origin. However, we will use function spaces of different regularity, which leads to some subtleties {in the construction of examples coming from de Rham complexes. In the two steps of our construction, we have to impose different regularity assumptions as discussed in Section \ref{sec:struct}. Moreover, we have to prove that the change of regularity assumptions does not influence the cohomology, see Lemma \ref{lem:change-reg}.}

Next, we denote by $Z^i$ the $i$-th column of \eqref{diag:multi-rows}, i.e.\ 
\begin{equation}\label{def:Zi}
Z^i:=Z^{i,0}\oplus Z^{i,1}\oplus\dots\oplus Z^{i,N}.
\end{equation}
 Each element in $Z^{i}$ can be viewed as a vector with the $j$-th component in $Z^{i, j}$.  Then any linear operator on $Z^{i}$ can be written in a matrix form. For example, the operators $K^{i,j}$ can be collected into $K^{i}: Z^{i}\to Z^{i}$, which has the matrix form 
$$
K^{i}:=\left (
\begin{array}{cccccc}
0 & K^{i, 1} & 0 & 0 & \cdots&0\\
0& 0 & K^{i, 2}  & 0 & \cdots&0\\
&&\cdots &\cdots&&\\
0& 0 &  0&0&\cdots&0  \\
\end{array}
\right ).
$$
Similarly, the operators $d^{i,j}$ are collected into $d^{i}: Z^{i}\to Z^{i+1}$ given by 
$$
d^{i}:=\left (
\begin{array}{cccccc}
d^{i, 0} & 0 & 0 & 0 & \cdots&0\\
0&  d^{i, 1}  &0 & 0 & \cdots&0\\
&&\cdots &\cdots&&\\
0& 0 &  0&0&\cdots&d^{i, N}  \\
\end{array}
\right ),
$$
and the $S^{i,j}$ are collected into $S^i:Z^i\to Z^{i+1}$ given by 
$$
S^{i}:=\left (
\begin{array}{cccccc}
0 & S^{{i}, 1} & 0 & 0 & \cdots&0\\
0& 0 & S^{{i}, 2}  & 0 & \cdots&0\\
&&\cdots &\cdots&&\\
0& 0 &  0&0&\cdots&0  
\end{array}
\right ).
$$
Observe that the identities derived above have simple expressions in these terms, {i.e.\ $d^{i+1}\circ d^i=0$, $S^iK^i=K^{i+1}S^i$, $S^i=d^iK^i-K^{i+1}d^i$, $S^id^i=-d^{i+1}S^i$, and $S^{i+1}\circ S^i=0$.}  In particular, $(Z^\bs,d^\bs)$ is a complex, that can simply be viewed as the direct sum of the rows of \eqref{diag:multi-rows}. In particular, the cohomology of this complex is just the direct sum of the cohomologies of the individual rows $(Z^{\bs,j},d^{\bs,j})$. 

\subsection{Twisted complex}\label{sec:twisted}
In the first step, we show that the differential of the complex $(Z^\bs,d^\bs)$ can be modified without changing the cohomology. For this step, we assume \eqref{SKKS} and \eqref{DKKD} as well as that all the operators $K^{i,j}$ and $S^{i,j}$ are bounded. This of course implies that also $K^i:Z^i\to Z^i$ and $S^i:Z^i\to Z^{i+1}$ are bounded.

Now define the twisted operator $d_{V}^{i}:=d^i-S^i:Z^{i}\to Z^{i+1}$, i.e.\ 
\begin{equation}\label{dV}
d_{V}^{i}:=\left (
\begin{array}{cccccc}
d^{i, 0} & -S^{i, 1} & 0 & 0 & \cdots&0\\
0& d^{i, 1} & -S^{i, 2}  & 0 & \cdots&0\\
&&\cdots &\cdots&&\\
0& 0 &  0&0&\cdots&d^{i, N}  \\
\end{array}
\right ).
\end{equation}
From \eqref{DSSD} and \eqref{SS}, we obtain $d_{V}^{i+1}\circ d_{V}^{i}=0$ for each $i$. The complex $(Z^{\bs}, d_{V}^{\bs})$, i.e.,
\begin{equation}\label{twisted-complex}
\begin{tikzcd}
\cdots \arrow{r}{}& Z^{i-1} \arrow{r}{d_{V}^{i-1}} & Z^{i} \arrow{r}{d_{V}^{i}} & Z^{i+1} \arrow{r}{}& \cdots,  
\end{tikzcd}
\end{equation}
is referred to as the {\it twisted complex}, or the $d_{V}$-complex.

We can now prove that the twisted complex $(Z^{\bs}, d_{V}^{\bs})$ is isomorphic $(Z^{\bs}, d^{\bs})$, so in particular, their cohomologies are isomorphic, too. This means that we construct isomorphisms $F^{i}: Z^{i}\to Z^{i}, ~i=0, 1, \cdots, n$ such that the following diagram commutes:
\begin{equation}\label{FdvF}
\begin{tikzcd}
Z^{i} \arrow{r}{d^{i}}\arrow[d, "F^{i} "] &Z^{i+1}\arrow[d, "F^{i+1} "]  \\
Z^{i} \arrow{r}{d_{V}^{i}}&Z^{i+1}
\end{tikzcd}
\end{equation}
In fact, define $F^{i}:=\exp(K^{i})=I+K^i+\frac{1}{2}(K^i)^2+\frac{1}{6}(K^i)^{3}+\cdots$, the matrix exponential of $K^{i}$, i.e., 
\begin{equation}\label{matrix-F}
F^{i}:=\left (
\begin{array}{cccccc}
I &K^{i, 1} & \frac{1}{2}K^{i, 1}K^{i, 2} & \frac{1}{6}K^{i, 1}K^{i, 2}K^{i, 3} & \cdots&\frac{1}{N!}K^{i, 1}K^{i, 2}\cdots K^{i, N}\\
0& I & K^{i, 2}  & \frac{1}{2}K^{i, 2}K^{i, 3} & \cdots&\frac{1}{(N-1)!}K^{i, 2}K^{i, 3}\cdots K^{i, N}\\
&&\cdots &\cdots&&\\
0& 0 &  0&0&\cdots&I  
\end{array}
\right ).
\end{equation}
Observe that $(K^i)^{N+1}=0$ by construction, so the exponential series actually is a finite sum. Moreover each $F^i$ is obviously invertible, with inverse given as $\exp(-K^i)$. Now we are ready to formulate our first main result.

\begin{theorem}\label{thm:input-twisted}
  Assume \eqref{SKKS} and \eqref{DKKD} and that all the operators $K^{i,j}$ and $S^{i,j}$ are bounded. Then the operators $F^i$ defined above are bounded and induce an isomorphism between the twisted complex \eqref{twisted-complex} and the complex $(Z^\bs,d^\bs)$. Hence  for the cohomology of \eqref{twisted-complex}, we get $\mathcal{H}^k(Z^\bs,d_V^\bs)\cong\oplus_{j=0}^N \mathcal{H}^k(Z^{\bs,j},{d^{\bs, j}})$. Here for a complex $(X^{\bs}, d^{\bs})$, we use $\mathcal{H}^{k}(X^{\bs}, d^{\bs})$ to denote the cohomology at index $k$.
\end{theorem}
 \begin{proof}
 Since we have verified that each $F^i$ is an isomorphism, it remains to prove that the diagram \eqref{FdvF} commutes for each $i$.  We first use induction to prove that 
\begin{equation}\label{dKm}
d^i(K^i)^{m}-(K^{i+1})^{m}d^i=mS^i(K^i)^{m-1}, \quad m\geq 1.
\end{equation} 
For $m=1$, this is \eqref{DKKD}. Assume that \eqref{dKm} holds for $m=\ell$ and let us leave out the index $i$, i.e., $dK^{\ell}-K^{\ell}d=\ell SK^{\ell-1}$. Then it suffices to verify \eqref{dKm} for $m=\ell+1$. In fact,
\begin{align*}
  dK^{\ell+1}-K^{\ell+1}d&=dK^{\ell}K-K^{\ell+1}d=K^\ell dK+\ell SK^\ell-K^{\ell+1}d\\
                         &=K^{\ell}S+\ell SK^{\ell}=(\ell+1)SK^\ell,
\end{align*}
where the second last equality is by \eqref{SKKS}. Now writing \eqref{dKm} as $K^md=dK^m-SmK^{m-1}$ and summing appropriately, we readily get $\exp(K)\circ d=(d-S)\circ\exp(K)$ as claimed. 
\end{proof}

\subsection{The BGG construction}\label{sec:BGG}

The second step of the construction starts from the twisted complex and constructs new differentials on certain subspaces of the spaces $Z^i$. We will then prove that the cohomology of the resulting complex is isomorphic to the cohomology of the twisted complex. In this second step, we will not need the operators $K^{i,j}$ any more, so we will only need to assume that we have bounded operators $S^{i,j}$ that satisfy \eqref{DSSD} and \eqref{SS}. In addition, we need that for each $i,j$, the range $\ran(S^{i,j})$ is closed in $Z^{i+1,j-1}$, which of course implies that $\ran(S^i)$ is closed in $Z^{i+1}$.  

  The starting point for this is that \eqref{SS} implies that also $S^\bs$ defines a differential on $Z^\bs$, so the ``diagonals'' in \eqref{diag:multi-rows} are complexes, too. (In the geometric BGG construction, these correspond to point-wise standard complexes for Lie algebra cohomology.) In particular, by the assumption on closed ranges, we obtain a decomposition
\begin{equation}\label{algebraic-hodge}
Z^{i, j}=\ran(S^{i-1, j+1})\oplus \ran(S^{i-1, j+1})^{\perp}=\ran(S^{i-1, j+1})\oplus \ker(S^{i, j})^{\perp}\oplus \Upsilon^{i, j},
\end{equation}
where 
\begin{equation}\label{def:Upsilon}
\Upsilon^{i, j}:=\ran(S^{i-1, j+1})^{\perp}\cap \ker(S^{i, j})
\end{equation}
 plays the role of cohomology or of harmonic forms in a ``Hodge decomposition''. In the geometric BGG theory, the decomposition \eqref{algebraic-hodge} is obtained without involving an inner product, here we have used the Hilbert space structures (reflected in $\ran(S^{i-1, j+1})^{\perp}$) to simplify the presentation. {In what follows, we will use the projections onto the three summands of the direct sum decomposition in \eqref{algebraic-hodge} and denote them by $P_{\ran(S)}$, $P_{\ker(S)}$  and $P_\Upsilon$. If we want to make the degrees explicit, we write $P_{\ran(S^{i, j})}$, $P_{\ker(S^{i, j})}$ and $P_{\Upsilon^{i,j}}$, respectively.} 
 
Another important ingredient for Hodge theory {are adjoints, in our setting we use partial inverses} of the $S$ operators. With respect to the decomposition \eqref{algebraic-hodge}, we define bounded operators $T^{i, j}: Z^{i, j}\to Z^{i-1, j+1}$ by 
\begin{equation}\label{def:T}
T^{i, j}:=(S^{i-1, j+1})^{-1}\circ P_{\ran(S^{i-1, j+1})},
\end{equation}
where $(S^{i-1, j+1})^{-1}: \ran(S^{i-1, j+1})\to \ker(S^{i-1, j+1})^{\perp}$ is the inverse of the isomorphism $\ker(S^{i-1, j+1})^{\perp}\to\ran(S^{i-1, j+1})$ defined by $S^{i-1,j+1}$. By definition, $\ran(T^{i, j})=\ker(S^{i-1, j+1})^{\perp}$ is also closed. Of course, the operators $T^{i,j}$ can be collected into a bounded operator $T^i:Z^i\to Z^{i-1}$ with a closed range. 

\begin{lemma}
The operators defined above satisfy the following properties for any
$i, j$:
\begin{equation}\label{TT}
T^{i-1, j+1}\circ T^{i, j}=0,
\end{equation}
\begin{equation}\label{TST}
T^{i, j}\circ S^{i-1, j+1}\circ T^{i, j}=T^{i, j},
\end{equation}
\begin{equation}\label{STS}
S^{i, j}\circ T^{i+1, j-1}\circ S^{i, j}=S^{i, j}.
\end{equation}
\begin{equation}\label{STperp}
\ran(S^{i, j})=\ker(T^{i+1, j-1})^{\perp}, \quad \ran(T^{i, j})=\ker(S^{i-1, j+1})^{\perp}.
\end{equation}
\end{lemma}
\begin{proof}
To prove \eqref{TT}, note that 
$$
T^{i-1, j+1}\circ T^{i, j}=(S^{i-2, j+2})^{-1}\circ P_{\ran(S^{i-2, j+2})}\circ (S^{i-1, j+1})^{-1}\circ P_{\ran(S^{i-1, j+1})}=0,
$$
since $(S^{i-1, j+1})^{-1}\circ P_{\ran(S^{i-1, j+1})}u\in  \ker(S^{i-1, j+1})^{\perp}\subset \ran(S^{i-2, j+2})^{\perp}$.

For \eqref{TST}, we have
\begin{align*}
T^{i, j}\circ S^{i-1, j+1}\circ T^{i, j}&=T^{i, j}\circ S^{i-1, j+1}\circ (S^{i-1, j+1})^{-1}\circ P_{\ran(S^{i-1, j+1})}\\&
=T^{i, j}\circ P_{\ran(S^{i-1, j+1})}=T^{i, j}.
\end{align*}
Similarly we can prove \eqref{STS}. 
\end{proof}
\begin{remark}\label{rmk:STbijective}
 By definition, if $S^{i, j}$ is bijective, then $T^{i+1,j-1}$ is just the inverse of $S^{i,j}$ and hence is bijective, too. Conversely, if $T^{i, j}$ is bijective, then from \eqref{TST} we have that $S^{i-1, j+1}$ is its inverse and thus also bijective.
\end{remark}

We next sketch the BGG construction. {The technical details are worked out in Section \ref{sect:cohomology} below, which also gives motivations for the definitions}.  By construction, $T^{i+1}d^{i}$ maps $Z^{i}$ to $Z^{i}$ but it ``moves each component down by a row'', i.e.\ it maps $Z^{i,j}$ to $Z^{i,j+1}$ and therefore is nilpotent. Thus we can define a bounded operator $G^{i}: Z^{i}\to Z^{i-1}$ by 
\begin{equation}\label{def:H}
G^{i}=-\sum_{k=0}^{\infty}(T^{i}d^{i-1})^{k}T^{i},
\end{equation}
where the sum in \eqref{def:H} is finite since $T^{i}d^{i-1}$ is nilpotent. Now define 
$$
\Upsilon^{i}:=\ran(S^{i-1})^{\perp}\cap \ker(S^{i}),
$$
 so this is the sum of the spaces $\Upsilon^{i,j}$ from \eqref{algebraic-hodge}. Here we use the convention that $\ran(S^{-1})^{\perp}=Z^{0}$ and $\ker(S^{N})=Z^{N}$. Further, define $A^{i}: \Upsilon^{i }\to Z^i$ by 
 \begin{equation}\label{def:A}
  A^{i}:=I-G^{i+1}d_{V}^{i}.
 \end{equation}
  As we shall verify below, $d_{V}^{i}A^{i}$ has values in $\ran(S^i)^{\perp}$, so using the projection from \eqref{algebraic-hodge} to define 
 \begin{equation}\label{def:D}
 D^{i}:=P_{\ker(S^{i+1})}d_{V}^{i}A^{i}=P_{\Upsilon^{i+1}}d_{V}^{i}A^{i},
 \end{equation}
  we get an operator $D^i:\Upsilon^i\to\Upsilon^{i+1}$. These operators form the BGG sequence
\begin{equation}\label{bgg-sequence}
\begin{tikzcd}
\cdots \arrow{r}&\Upsilon^{i-1}\arrow{r}{D^{i-1}}&\Upsilon^{i}\arrow{r}{D^{i}}&\Upsilon^{i+1}\arrow{r}{}&\cdots,
\end{tikzcd}
\end{equation}
We will see that \eqref {bgg-sequence} is a complex, that $A^{\bs}$ defines a morphism of complexes from \eqref{bgg-sequence} to the twisted complex, and that this induces an isomorphism in cohomology. The explicit matrix form of some operators in this section will be given in an appendix.

Before presenting the proof, we consider the example of a diagram with two rows.
 This is the setting of \cite{arnold2021complexes}, but we do not assume any injectivity/surjectivity conditions. For $N=2$, we have $G^{i}=-T^{i}$. Therefore
$$
Gd_{V}=\left ( 
\begin{array}{cc}
0 & 0\\
-T & 0 
\end{array}
\right )
\left ( 
\begin{array}{cc}
d & -S\\
0 & d 
\end{array}
\right )=\left ( 
\begin{array}{cc}
0 & 0\\
-Td & TS 
\end{array}
\right )=
\left ( 
\begin{array}{cc}
0 & 0\\
-Td & P_{\ker(S)^{\perp}}
\end{array}
\right ),
$$
and
$$
D= P_{\ker(S)}d_{V}A=
P_{\ker(S)}\left ( 
\begin{array}{cc}
d & -S\\
0 & d
\end{array}
\right )
\left ( 
\begin{array}{cc}
I & 0\\
Td & P_{\ker(S)}
\end{array}
\right )=
\left ( 
\begin{array}{cc}
 P_{\ran(S)^\perp} d & 0\\
 P_{\ker(S)}dTd & dP_{\ker(S)}
\end{array}
\right ).
$$

The BGG complex \eqref{bgg-sequence} is simplified with $\Upsilon^{i}=\ran(S^{i-1})^{\perp}\cap \ker(S^{i})$, and for $(\alpha_0, \alpha_1)\in \Upsilon^{i}$, 
\begin{equation}\label{BGG-operator}
 {D}^{i}\left (
\begin{array}{c}
\alpha_{0}\\
\alpha_{1}
\end{array}
\right )
=\left ( 
\begin{array}{c}
P_{\ran(S)^{\perp}}d\alpha_{0}\\
d\alpha_{1}+P_{\ker(S)}dTd\alpha_{0}
\end{array}
\right ).
\end{equation}
We have simplified the operator using the fact that $\alpha_{1}\in \ker(S^{i})$.

 Next we show how this formula simplifies in the presence of injectivity/surjectivity conditions. 
The discussions will involve three $S$ operators indicated in the following diagram
\begin{equation}\label{2row-diagram}
\begin{tikzcd}
\cdots \arrow{r} & Z^{i, 0}\arrow{r}{d^{i, 0}} &Z^{i+1, 0}  \arrow{r}{}&\cdots\\
\cdots \arrow{r}{d^{i-1, 1}} \arrow[ur, "{S^{i-1, 1}}"]& Z^{i, 1}\arrow{r}{d^{i, 1}} \arrow[ur, "{S^{i, 1}}"]&Z^{i+1, 1}\arrow[ur, "{S^{i+1,1}}"] \arrow{r}{}&\cdots,
 \end{tikzcd}
\end{equation}
and there are three cases.
\begin{enumerate}
\item Assume that $S^{i, 1}$ and $S^{i+1, 1}$ are injective. In this case $\ker(S^{i, 1})=0$, so $\alpha_{1}$ is eliminated from the system. The operator $ {D}^{i}$ is reduced to 
$$
 {D}^{i}\left (
\begin{array}{c}
\alpha_{0}\\
0
\end{array}
\right ):=\left ( 
\begin{array}{c}
P_{\ran(S)^{\perp}}d\alpha_{0}\\
0
\end{array}
\right ).
$$
\item Assume that $S^{i,1}$ is bijective. In this case $\ker(S^{i, 1})=0$, so $\alpha_{1}$ is eliminated from the system. 
The operator $ {D}^{i}$ is reduced to 
$$
 {D}^{i}\left (
\begin{array}{c}
\alpha_{0}\\
0
\end{array}
\right ):=\left ( 
\begin{array}{c}
0\\
 dTd\alpha_{0}
\end{array}
\right ).
$$
Note that we have removed $P_{\ker(S)}$ in front of $dTd\alpha_{0}$. From Remark \ref{rmk:STbijective}, $T=S^{-1}$ is bijective. Therefore $S(dTd)=-dSTd=0$. This implies that $dTd\alpha$ is automatically in $\ker(S)$.
\item Assume that  $S^{i-1,1}$ and $S^{i,1}$ are surjective. In this case $\ran(S)^{\perp}=0$, so $\alpha_{0}$ is eliminated from the system. The operator ${D}^{i}$ is reduced to 
$$
 {D}^{i}\left (
\begin{array}{c}
0\\
\alpha_{1}
\end{array}
\right ):=\left ( 
\begin{array}{c}
0\\
d\alpha_{1}
\end{array}
\right ).
$$
\end{enumerate}
The three differential operators in the elasticity complex in three dimensions \cite{arnold2021complexes} correspond to these three cases, respectively. 

\subsection{Cohomology of the BGG sequence}\label{sect:cohomology}
We next add the details for the construction of the sequence \eqref{bgg-sequence}, prove that it is a complex and show that it computes the same cohomology as the twisted complex \eqref{twisted-complex}. We prove this fact by exploiting the properties of $G^{i}$ and $A^{i}$. The explicit cohomology-preserving cochain maps will be given in the proof of Theorem \ref{thm:twisted-bgg} below. 

First we can verify the following properties of $G^{i}$:
\begin{equation}\label{H-defining1}
G^{i}|_{\ker(T^{i})}=0,
\end{equation}
\begin{equation}\label{H-defining2}
\phi-d_{V}^{i-1}G^{i}\phi\in \ker(T^{i}),
\end{equation}
\begin{equation}\label{H-defining3}
G^{i}\phi\in \ran(T^{i}).
\end{equation}
In fact, \eqref{H-defining1} and \eqref{H-defining3} are by the definition \eqref{def:H}. It suffices to verify \eqref{H-defining2}. To this end, we note that 
$$
T^{i}d_{V}^{i-1}T^{i}=T^{i}(d^{i-1}-S^{i-1})T^{i}=T^{i}d^{i-1}T^{i}-T^{i}S^{i-1}T^{i}=(-I+T^{i}d^{i-1})T^{i}.
$$
This implies that on $\ran(T^{i})$, we get  $T^{i}d_V^{i-1}=-I+T^{i}d^{i-1}$. Since $\ran(G)\subset \ran(T)$, we have $T^{i}d_{V}^{i-1}{G}^{i}=(-I+T^{i}d^{i-1})G^{i}=T^{i}$.

Indeed, property \eqref{H-defining2} provides the motivation for the construction of $G^i$: As we have noted above, restricting $T^id_V^{i-1}$ to $\ran(T^i)$, one obtains a map $\ran(T^i)\to\ran (T^i)$ that is nilpotent and hence {$I-T^id_V^{i-1}$ is invertible on that subspace}. Composing the inverse with $-T$, one obtains a map $Z^i\to Z^{i-1}$ that satisfies \eqref{H-defining2}, and $G^i$ is simply obtained from writing the inverse as a Neumann series.    

From   \eqref{H-defining1}-\eqref{H-defining3}, we have the following lemma.
\begin{lemma}\label{lem:proof1}
\begin{enumerate}
\item For any $\alpha\in \Upsilon^{i}$, we get $A^i\alpha\in\ran(S^{i-1})^\perp$ and $P_{\ker(S^i)}A^i\alpha=P_{\Upsilon^i}A^i\alpha=\alpha$. 
  Moreover, we get $d_{V}^{i}A^{i}(\alpha)\in \ran(S^{i})^{\perp}$.
\item Conversely, suppose that $\psi\in \ran(S^{i-1})^{\perp}$ has the property that $d_{V}^{i}\psi\in \ran(S^i)^{\perp}$. Then $\psi=A^{i}(\beta)$, where $\beta=P_{\ker(S^i)}\psi=P_{\Upsilon^i}\psi$.
\end{enumerate}
\end{lemma}
\begin{proof}
\begin{enumerate}
\item Since $G^{i+1}$ has values in $\ker(S^{i})^{\perp}$, the first two properties follow readily. Moreover, $d_{V}^{i}A^{i}(\alpha)=d_{V}^{i}\alpha-d_{V}^{i}G^{i+1}d_{V}^{i}\alpha$, so applying \eqref{H-defining2} to $d_{V}^{i}\alpha$, the last property follows.
\item From part 1, we conclude that $\psi-A^{i}(\beta)$ lies in $\ker(S^{i})^{\perp}$ and satisfies $d_{V}^{i}(\psi-A^{i}(\beta))\in \ran(S^{i})^{\perp}$. Thus it suffices to prove that $\psi\in \ker(S^{i})^{\perp}$ and $d_{V}^{i}\psi\in \ran(S^{i})^{\perp}$ imply $\psi=0$. But the latter condition says that $d^{i}\psi-S^{i}\psi\in \ran(S^{i})^{\perp}$. Thus the component of $d^{i}\psi$ in ${\ran(S^{i})}$ is ${S^{i}(\psi)}$. Writing $\psi=(\psi^0, \psi^1, \cdots, \psi^N)$, $\psi\in \ker(S^{i})^\perp$ implies $\psi^0=0$.  Then $d^{i}\psi=(0, d^{i, 1}\psi^1, \cdots, d^{i, N}\psi^N)$ while $S^{i}\psi=(S^{i, 1}\psi^1,  \cdots, S^{i, N-1}\psi^{N-1}, 0)$. Comparing the first component, we get $S^{i, 1}\psi^1=0$ and thus $\psi^1=0$ since $\psi^1\in \ker(S^{i, 1})^\perp$. Similarly, we get $S^{i, 2}\psi^2=0$ and thus $\psi^2=0$. Repeating this argument, we see that $\psi=0$. 
\end{enumerate}
\end{proof}

  Using this, we can formulate our second main result.

  \begin{theorem}\label{thm:twisted-bgg} Assume that we start from the twisted complex \eqref{twisted-complex} associated to \eqref{diag:multi-rows} with bounded operators $S^{i,j}$ {with closed range} such that \eqref{DSSD} and \eqref{SS} are satisfied. Then the BGG sequence \eqref{bgg-sequence} is a complex, whose cohomology is isomorphic to the cohomology of the twisted complex \eqref{twisted-complex} (and thus isomorphic to the cohomology of the input by Theorem \ref{thm:input-twisted}). The map $A^i$ and the map $B^{i}$ defined by \eqref{def:B} in the proof are complex maps that induce inverse isomorphisms in cohomology. 
\end{theorem}

\begin{proof}
  Part 1 of Lemma \ref{lem:proof1} shows that $D^{i}\alpha$ is the component of $d_V^{i}A^{i}\alpha$ in $\Upsilon^{i+1}$. Since $d_V^{i}A^{i}\alpha\in \ran(S^{i})^\perp$ and $d_V^{i+1}d_V^{i}A^{i}\alpha=0$, part (2) of Lemma \ref{lem:proof1} shows that
  \begin{equation}\label{dVAAD}
    d_V^{i}A^{i}\alpha=A^{i+1}D^{i}\alpha.
  \end{equation}
  This in turn readily implies $d_V^{i+1}A^{i+1}D^{i}\alpha=0$ and hence $D^{i+1}D^i\alpha=0$. Hence the BGG sequence is a complex and \eqref{dVAAD} shows that $A^\bs$ defines a morphism of complexes from the BGG complex to the twisted complex. 

Now for a form $\phi$ consider $\tilde{\phi}:=\phi-d_V^{i-1}G^{i}\phi$. By \eqref{H-defining2}, this lies in $\ran(S^{i-1})^\perp$ and we define $B^{i}(\phi)$ to be its component in $\Upsilon^{i}$, i.e., 
\begin{equation}\label{def:B}
B^{i}\phi:=P_{\Upsilon^{i}}(I-d_{V}^{i-1}G^{i})\phi,
\end{equation}
where $P_{\Upsilon^{i}}$ is the projection to $\ran(S^{i-1})^{\perp}\cap \ker(S^{i})$.
 Now consider 
$$
\psi:=\tilde{\phi}-G^{i+1}d_V^{i}\tilde{\phi}=\phi-d_V^{i-1}G^{i}\phi-{G}^{i+1}d_V^{i}\phi.
$$
Then $\psi$ lies in $\ran(S^{i-1})^\perp$ and $d_V^{i}\psi=d_V^{i}\phi-d_V^{i}{G}^{i+1}d_V^{i}\phi$, which as above, lies in $\ran(S^{i})^\perp$ by \eqref{H-defining2}. Moreover, the component of $\psi$ in $\Upsilon^{i}$ still equals $B^{i}\phi$, so by part (2) of Lemma \ref{lem:proof1}, we get 
\begin{equation}\label{ABphi}
A^{i}B^{i}(\phi)=\psi=\phi-d_V^{i-1}{G}^{i}\phi-{G}^{i+1}d_V^{i}\phi.
\end{equation}
Now $B^{i+1}d_V^{i}\phi$ is the component in $\Upsilon^{i+1}$ of $d_V^{i}\phi-d_V^{i}{G}^{i+1}d_V^{i}\phi=d_V^{i}A^{i}B^{i}\phi$. This shows that $B^{i+1}d_V^{i}\phi=D^{i}B^{i}\phi$, so $B$ is a morphism of complexes, too. Equation \eqref{ABphi} then shows that $AB$ is chain homotopic to the identity. On the other hand, since $A^{i}\alpha=\alpha-{G}^{i+1}d_V^{i}\alpha\in \ran(S^{i-1})^\perp$, we get ${G}^{i}A^{i}\alpha=0$. Hence $B^{i}A^{i}\alpha$ simply is the component in $\Upsilon^{i}$ of $A^{i}\alpha$ which equals $\alpha$. Thus $B^{i}A^{i}=I$ and we see that $B$ and $A$ induce isomorphisms in cohomology which are inverse to each other. 
\end{proof}

  \section{Examples}\label{sec:examples}
  
  We demonstrate several examples of the abstract framework. In particular, these include some complexes in \cite{beig2020linearised} with applications in general relativity. Following the general framework, for each example we should specify the diagram \eqref{diag:multi-rows}, i.e., the spaces $Z^{i, j}$, the operators $d^{i, j}$, $K^{i, j}$ (and thus $S^{i, j}$ defined by \eqref{DKKD}), such that the assumption \eqref{SKKS} holds.  We'll say more about how these operators are obtained and the relation to the constructions in \cite{vcap2001bernstein} (that apply in a smooth setting) in Section \ref{sec:background} below.
  As we shall see below, to derive complexes with more general function spaces (e.g.,  Sobolev spaces), these algebraic structures are not enough due to regularity issues. With a similar idea as \cite{arnold2021complexes}, this subtlety caused by the regularity of functions can be fixed by requiring in addition that the input complexes have a uniform representation of cohomology, not depending on the regularity. For de Rham sequences, this boils down to the results by Costabel and McIntosh \cite{costabel2010bogovskiui}. However, the argument given here is different from the proof in \cite{arnold2021complexes} based on Poincar\'e operators.

  In Section \ref{sec:BGG}, we have seen that for diagrams with two rows and assuming injective/surjectivity conditions on the $S$-operators, we recover the examples from \cite{arnold2021complexes}, so we focus on diagrams with at least three rows.

  Throughout the rest of this paper, we assume that $U$ is a bounded Lipschitz domain in $\mathbb{R}^{n}$. Thus the results in \cite{costabel2010bogovskiui} {apply} and in particular there exists a uniform representation of cohomology for the de Rham complex, c.f.\ \cite[Theorem 1.1]{costabel2010bogovskiui}.

 \subsection{The structure of the examples}\label{sec:struct}
    To set up one of our examples, we fix finite dimensional vector spaces $\mathbb V_j$ and real numbers $q_j$ for $j=0,\dots,N$, and the examples apply to bounded Lipschitz domains $U\subset\mathbb R^n$, mostly with $n=3$. The spaces $Z^{i,j}$ can be interpreted as differential $i$-forms on $U$ with values in $\mathbb V_j$ or equivalently as functions to $\Lambda^i\mathbb R^{n*}\otimes\mathbb V_j$ of appropriate regularity. The operators we construct all make sense in the smooth setting and the identities can usually be obtained respectively checked there, but the main case to consider is that the functions   in $Z^{i,j}$ lie in the Sobolev space $H^{q_j-i}$. The operators $d^{i,j}:Z^{i,j}\to Z^{i+1,j}$ are all induced by the exterior derivative, so since the regularity in the target is always one lower than in the source, they are bounded. The regularities $q_j$ depend on some fixed number $q\in\mathbb R$, but we'll need different choices for the two steps and a key lemma to connect the two steps. In the first step we will take $q_j=q$ for all $j$, while in the second step $q_j=q-j$ for each $j$.

    To start with the first step, we take $q_j=q$ for all $j$, and hence $Z^{i,j}=H^{q-i}\otimes\Lambda^i\mathbb R^{n*}\otimes\mathbb V_j$. The operators $K^{i,j}$ will always be built up from the operators of multiplication by bounded smooth functions, so since the regularity in $Z^{i,j}$ and $Z^{i,j-1}$ are the same, they define bounded operators $Z^{i,j}\to Z^{i,j-1}$ for each $i$ and $j$. One next verifies directly that $S^{i,j}:=d^{i,j-1}K^{i,j}-K^{i+1,j}d^{i,j}$ in the smooth case satisfies \eqref{SKKS}, so this extends to the Sobolev setting, and all the assumptions for Theorem \ref{thm:input-twisted} are satisfied. Thus we conclude that the complex $(Z^\bs,d_V^\bs)$ is isomorphic to the sum of the Sobolev de Rham complexes with values in $\mathbb V_j$, and hence its cohomology is well understood. 

    To pass to the second step, we need an additional observation: In each of the examples, one verifies directly that  the operator $S^{i,j}$ defined by \eqref{DKKD} maps a smooth  function $\phi$ to $\partial^{i,j}\circ\phi$, where $\partial^{i,j}:\Lambda^i\mathbb R^{n*}\otimes\mathbb V_j\to \Lambda^{i+1}\mathbb R^{n*}\otimes\mathbb V_{j-1}$ is a linear map. Hence also in the Sobolev setting $S^{i,j}$ is given by composition with $\partial^{i,j}$. This readily shows that $S^{i,j}$ also extends as a bounded linear operator $\tilde Z^{i,j}\to\tilde Z^{i+1,j-1}$ where $\tilde Z^{i,j}:=H^{q-i-j}\otimes\Lambda^i\mathbb R^{n*}\otimes\mathbb V_j$. The identities \eqref{DSSD} and \eqref{SS} of course extend to this regularity since they hold on smooth forms. Moreover, $\ran(S^{i,j})\subset \tilde Z^{i+1,j-1}$ is the subspace $H^{q-i-j}\otimes\ran(\partial^{i,j})$, which is evidently closed in $\tilde Z^{i+1,j-1}$ (and the change of regularity is needed to get this). Hence all conditions needed to apply Theorem \ref{thm:twisted-bgg} are satisfied, and it remains   to understand the cohomology of $(\tilde Z^\bs,d_V^\bs)$. This is described by the following key lemma, which crucially depends on the existence of uniform smooth representatives for the Sobolev-de Rham cohomology. To formulate this, let us write $Z_q^\bs$ for the complex in regularity $q$, {i.e., $Z_{q}^{i, j}:=H^{q-i}\otimes\Lambda^i\mathbb R^{n*}\otimes\mathbb V_j$.} Then by definition $Z_q^{i,j}\subset\tilde Z^{i,j}\subset Z_{q-N}^{i,j}$, so we have inclusions $\mu:Z_q^\bs\to \tilde Z^\bs$ and $\nu:\tilde Z^\bs\to Z_{q-N}^\bs$.  

    \begin{lemma}\label{lem:change-reg}
 Assume that $U$ is a bounded Lipschitz domain. Then
      in the notation introduced above, both $\mu:(Z_q^\bs,d_V^\bs)\to (\tilde Z^\bs,d_V^\bs)$ and $\nu:(\tilde Z^\bs,d_V^\bs)\to (Z_{q-N}^\bs,d_V^\bs)$ are inclusions of subcomplexes that induce isomorphisms in cohomology.  
    \end{lemma}
    \begin{proof} 
      By construction, both $\mu$ and $\nu$ are compatible with the differentials $d_V$ and hence induce maps in cohomology. The composition $\nu\circ\mu$ is just the inclusion $(Z_q^\bs,d^\bs_V)\to (Z_{q-N}^\bs,d_V^\bs)$. Via the isomorphisms $F$ from Theorem \ref{thm:input-twisted}, this is conjugate to the inclusion $(Z_q^\bs,d^\bs)\to (Z_{q-N}^\bs,d^\bs)$ of a sum of Sobolev-de Rham complexes. The latter induces isomorphisms in cohomology by \cite{costabel2010bogovskiui}, and this carries over to the $d_V$ complexes via $F$. This implies that $\mu$ induces an injection in cohomology, while $\nu$ induces a surjection in cohomology and to complete the proof, it suffices to show that the map in cohomology induced by $\nu$ is injective.

      To prove this, we have to take forms $\psi_j\in\tilde Z^{i,j}$ for some $i$ and $j=0,\dots ,N$ such that there are forms $\phi_j\in Z_{q-N}^{i-1,j}$ such that $d_V(\phi_0,\dots,\phi_N)=(\psi_{0},\dots,\psi_N)$ (which of course implies that $d_V(\psi_{0},\dots,\psi_N)=0$). We then have to show that we can also find forms $\tilde\phi_j\in\tilde Z^{i-1,j}$ such that $d_V(\tilde\phi_0,\dots,\tilde\phi_N)=(\psi_{0},\dots,\psi_N)$, i.e.~we have to improve the regularity of the $j$th component from {$H^{q-N+1}$ to $H^{q-j+1}$}. In particular, there is nothing to do for $j=N$ and we use this as the starting point of a backwards induction. So take $j<N$ and assume that $\phi_k\in\tilde Z^{i,k}$ for $k=j+1,\dots,N$. Then by assumption, the $j$th component of $d_V(\phi_0,\dots,\phi_N)=(\psi_{0},\dots,\psi_N)$ reads as $\psi_j=d\phi_j-S\phi_{j+1}$. By the inductive hypothesis, $\phi_{j+1}\in\tilde Z^{i-1,j+1}$ so $S\phi_{j+1}\in\tilde Z^{i,j}$. Hence $\psi_j+S\phi_{j+1}$ lies in $H^{q-j}$ and by assumption $\psi_j+S\phi_{j+1}=d\phi_j$. But this means that $\psi_j+S\phi_{j+1}$ is exact in the sum of Sobolev de Rham complexes in regularity $q-N$, so it is also exact in regularity $q-j$ by \cite{costabel2010bogovskiui}. Hence there is a form $\hat\phi_j\in \tilde Z^{i-1,j}$ such that $\psi_j+S\phi_{j+1}=d\hat\phi_j$. Now $d(\hat\phi_j-\phi_j)=0$, so by \cite{costabel2010bogovskiui}, we can write $\hat\phi_j=\phi_j+\alpha+d\beta$, where $\alpha$ is a smooth representative for the cohomology class of $\hat\phi_j-\phi_j$ and $\beta\in Z_{q-N}^{i-2,j}$. Now defining $\tilde\phi_j:=\hat\phi_j-\alpha\in \tilde Z^{i-1,j}$, we get $d\tilde\phi_j=d\hat\phi_j=\psi_j+S\phi_{j+1}$ and $S\tilde\phi_j=S\phi_j-dS\beta$. Hence if we in addition replace $\phi_{j-1}$ by $\tilde\phi_{j-1}=\phi_{j-1}-S\beta$ and put $\tilde\phi_\ell=\phi_\ell$ for $\ell\neq j,j-1$, we get $d_V(\tilde\phi_0,\dots,\tilde\phi_N)=(\psi_0,\dots,\psi_N)$. This completes the inductive step and hence the proof.
    \end{proof}

    We will usually work in dimension $n=3$ and to be more compatible with the literature, we often replace differential forms by their vector proxies and the exterior derivative by gradient, curl and divergence. Following the notation in \cite{arnold2021complexes}, we introduce notation for the algebraic operations we need in this picture.
 {   Our notations for vector/matrix proxies and basic linear algebraic operations are summarized as follows:
\begin{table}[h!]
\begin{center}
\begin{tabular}{c|c}

$\mathbb V$ & $\mathbb R^n$ (also for a representation when there is no danger of confusion)\\
$\mathbb M$ &the space of all $n\times n$-matrices\\
$\mathbb S$ & symmetric matrices\\
$\mathbb K$ & skew symmetric matrices\\
$\mathbb T$ & trace-free matrices\\
$\skw: \M\to \K$ & skew symmetric part of a matrix\\
$\sym: \M\to \S$ & symmetric part of a matrix\\
$\tr:\M\to\R$ & matrix trace\\
$\iota: \R\to \M$  & the map $\iota u:= uI$ identifying a scalar with a scalar matrix\\
$\dev:\mathbb{M}\to \mathbb{T}$ & deviator (trace-free part of a matrix) given by $\dev u:=u-1/n \tr (u)I$\\

\end{tabular}
\end{center}
\end{table}}

Moreover, in three space dimensions, we can identify a skew symmetric matrix with a vector,
$$
 \mskw\left ( 
\begin{array}{c}
v_{1}\\ v_{2}\\ v_{3}
\end{array}
\right ):= \left ( 
\begin{array}{ccc}
0 & -v_{3} & v_{2} \\
v_{3} & 0 & -v_{1}\\
-v_{2} & v_{1} & 0
\end{array}
\right ).
$$
Consequently, we have   $\mskw(v)w = v\times w$ for $v,w\in\V$.  We also define
$\vskw=\mskw^{-1}\circ \skw: \M\to \V$. We further define the Hessian operator $\hess:=\grad\circ\grad$ and {for any matrix valued function $u$, $\mathcal{S}u:=u^{T}-\tr(u)I$.}  

{ 
In vector/matrix proxies, we often deal with (co)vector-valued differential forms. Our convention of notation is that each column of a matrix corresponds to a form. For example, let $w=w_{i}dx^{i}$ be a vector-valued 1-form, where $w_{i}$ is a vector for $i=1, 2, 3$. Then the three rows of the matrix proxy of $w$ are $w_{1}$, $w_{2}$ and $w_{3}$, respectively. This is consistent with the convention in \cite{arnold2021complexes}. With this convention, the proxies of exterior derivatives ($\grad$, $\curl$, $\div$, etc.) act column-wise, while the $K$ operators act row-wise, as they will be defined on the vector values of forms (see examples below).}
  
  \subsection{Conformal Hessian complex}\label{sec:conf-Hess}

  In the language of Section \ref{sec:struct}, we put $n=3$, $N=2$, $\mathbb V_0=\mathbb V_2=\mathbb R$ and $\mathbb V_1=\mathbb R^3$. So elements in $Z^{i,0}$ and $Z^{i,2}$ are just differential forms of degree $i$, which we denote by $\phi$ and $\omega$, respectively. An element of $Z^{i,1}$ can be written as a triple $(\psi_1,\psi_2,\psi_3)$ of $i$-forms. In this notation, we define $K^{i,1}(\psi_1,\psi_2,\psi_3):=\sum_\ell x^\ell\psi_\ell$, where $x^1,x^2,x^3$ are the coordinate functions on $\mathbb R^3$. Likewise, we define $K^{i,2}(\omega):=({x^1\omega,x^2\omega,x^3\omega})$. Defining $S^{i,j}$ by \eqref{DKKD}, we obtain
      \begin{align*}
        S^{i,1}(\psi_1,\psi_2,\psi_3)&=\textstyle\sum_\ell(d(x^\ell\psi_\ell)-x^\ell d\psi_\ell)=\sum_\ell dx^\ell\wedge\psi_\ell,\\
        S^{i,2}\omega&=(dx^1\wedge\omega,dx^2\wedge\omega,dx^3\wedge\omega).
      \end{align*}
      In particular, in the picture of functions, the $S^{i,j}$ can be written as compositions with linear maps $\partial^{i,j}$ as discussed in Section \ref{sec:struct}.   
      
      There is only one relevant instance of equation \eqref{SKKS}, namely $K^{i+1,1}S^{i,2}=S^{i,1}K^{i,2}$, and both sides evidently map $\omega$ to $\sum_\ell x^\ell dx^\ell\wedge\omega$. Hence we have made all the verifications needed to apply our machinery. Converting to vector proxies, the BGG diagram has the form 
 \begin{equation}\label{conformal-hess-3rows}
\begin{tikzcd}
0 \arrow{r} &H^{q_0}\otimes \mathbb{R}  \arrow{r}{\grad} &H^{q_0-1}\otimes \mathbb{V} \arrow{r}{\curl}\arrow[dl, "I", shift left=1] &H^{q_0-2}\otimes \mathbb{V} \arrow{r}{\div}\arrow[dl, "\mskw", shift left=1] & H^{q_0-3}\otimes \mathbb{R} \arrow{r}{}\arrow[dl, "1/3\iota", shift left=1] & 0\\
0 \arrow{r}&H^{q_1}\otimes \mathbb{V}\arrow{r}{\grad} \arrow[u, "\cdot x"]\arrow[ur, "I"]&H^{q_1-1}\otimes \mathbb{M}  \arrow{r}{\curl} \arrow[ur, "2\vskw"]\arrow[u, "\cdot x"]\arrow[dl, "1/3\tr", shift left=1]&H^{q_1-2}\otimes \mathbb{M} \arrow{r}{\div}\arrow[ur, "\tr"] \arrow[u, "\cdot x"]\arrow[dl, "-2\vskw", shift left=1]&\arrow[u, "\cdot x"] H^{q_1-3}\otimes \mathbb{V} \arrow{r}{} \arrow[dl, "I", shift left=1]& 0\\
0 \arrow{r} &H^{q_2}\otimes \mathbb{R}\arrow[u, " \otimes x"]\arrow{r}{\grad} \arrow[ur, "\iota"]&H^{q_2-1}\otimes \mathbb{V}  \arrow[u, " \otimes x"]\arrow{r}{\curl} \arrow[ur, "-\mskw"]\arrow[u, " \otimes x"]&H^{q_2-2}\otimes \mathbb{V}\arrow[u, " \otimes x"] \arrow{r}{\div}\arrow[ur, "I"] & H^{q_2-3}\otimes \mathbb{R} \arrow[u, " \otimes x "]\arrow{r}{} & 0.
 \end{tikzcd}
\end{equation}     
  
 Here the maps $\cdot x$ send {an $\Bbb R^3$-valued} function $f$ to the function $x\mapsto { f(x)\cdot x}$ while $\otimes x$ maps a real valued function $g$ to the function $x\mapsto g(x)x$. { Then such operations are extended in a row-wise manner. For example, if $u$ is a vector, then $ u\otimes x$ is a matrix whose $i$-th row is the vector $u^{i}x$.} The operators in the twisted complex in this picture then read as 
$$  
d_{V}^{0}=
   \left (
\begin{array}{ccc}
\grad & -I & 0\\ 
0 & \grad & -\iota\\
0 & 0 & \grad
\end{array}
\right ), \quad 
d_{V}^{1}=
   \left (
\begin{array}{ccc}
\curl & -2\vskw & 0\\ 
0 & \curl & \mskw\\
0 & 0 & \curl
\end{array}
\right ), \quad
d_{V}^{2}=
   \left (
\begin{array}{ccc}
\div & -\tr & 0\\ 
0 & \div & -I\\
0 & 0 & \div
\end{array}
\right ). 
$$

  To pass to the BGG complex, we first move to $\tilde Z^{i,j}$ as discussed in Section \ref{sec:struct}. Then we just have to understand the complements to the image of each $S$ operators in the kernel of the next $S$ operators, and this just happens point-wise. From \eqref{conformal-hess-3rows} we can immediately read off that all operators $S^{i,1}$ are surjective and $S^{0,1}$ is bijective, while all $S^{i,2}$ are injective and $S^{2,2}$ is bijective. This {immediately provides the spaces in the resulting complex, while the operators can be derived from formula \eqref{D-matrix} in the appendix, and we get}
 \begin{equation}\label{conformal-hess}
\begin{tikzcd}
0 \arrow{r} &H^{q}\otimes \mathbb{R}  \arrow{r}{\dev\hess} &H^{q-2}\otimes (\mathbb{S}\cap \mathbb{T}) \arrow{r}{ \sym\curl} &H^{q-3}\otimes (\mathbb{S}\cap \mathbb{T}) \arrow{r}{\div\div} & H^{q-5}\otimes \mathbb{R} \arrow{r}{} & 0,
 \end{tikzcd}
 \end{equation}
 which is known as the \textit{conformal Hessian complex}.

 \smallskip
 
 In the language of differential forms, we can also show that our construction is compatible with Euclidean motions in an appropriate sense. Let $f$ be a Euclidean motion, so $f(x)=Ax+b$ for an orthogonal matrix $A=(a^i_j)$ and a vector $b\in\Bbb R^3$. Then for a bounded Lipschitz domain $U$, also $f(U)\subset\mathbb R^3$ is a bounded Lipschitz domain, and we have the standard pullback operator on differential forms (of any regularity), which we denote by $f^*$. Now in the notation above, for forms $\phi\in Z^{i,0}(f(U))$ and $\omega\in Z^{i,2}(f(U))$ we use this standard pullback $f^*\phi\in Z^{i,0}(U)$ and $f^*\omega\in Z^{i,2}(U)$. On $Z^{i,1}$, we need a different version of the pullback. Namely, representing an element of $Z^{i,1}(f(U))$ as a triple $(\psi_1,\psi_2,\psi_3)$ we define
   $$
f^*(\psi_1,\psi_2,\psi_3):=(\textstyle{\sum}_ja_1^jf^*\psi_j,\textstyle{\sum}_ja_2^jf^*\psi_j,\textstyle{\sum}_ja_3^jf^*\psi_j). 
$$
This may look arbitrary, the conceptual explanation is that we should actually view elements of $Z^{i,1}$ as $i$-forms with values in the tangent bundle and this is the natural action of a diffeomorphism on such forms. See Section \ref{sec:background} for more details.

The standard compatibility of a pullback of differential forms with the exterior derivative together with the fact that the $a^i_j$ are constants, implies that the pullback we have defined commutes with the exterior derivatives $d^i$ on all spaces in question. Moreover, we can verify by a direct computation that it is also compatible with the operators $S^{i,j}$. By definition, $f(x)^i=\sum_ja^i_jx^j+b^i$, where the $b^j$ are the components of $b$, which implies that $f^*dx^i=\sum_ja^i_jdx^j$ for any $i$. Inserting this, one immediately computes that for $(\psi_1,\psi_2,\psi_3)\in Z^{i,1}(f(U))$ both $S^{i,1}(f^*(\psi_1,\psi_2,\psi_3))$ and $f^*(S^{i,1}(\psi_1,\psi_2,\psi_3))$ are given by $\sum_{j,\ell}{ a_j^\ell} dx^j\wedge f^*\psi_\ell$. On the other hand, for $\omega\in Z^{i,2}(f(U))$ the $k$th component of $S^{i,2}f^*\omega$ is given by $dx^k\wedge f^*\omega$, while for the $k$th component of $f^*(S^{i,2}\omega)$, we obtain
$$
\textstyle\sum_j a^j_k f^*(dx^j\wedge\omega)=\textstyle\sum_{j,\ell}a^j_ka^j_\ell dx^\ell\wedge f^*\omega. 
$$
But since $A^t=A^{-1}$ we get $\sum_ja^j_ka^j_\ell=\delta_{k,\ell}$, the Kronecker delta, so we again get $dx^k\wedge f^*\omega$.

 This readily implies that the pullback is compatible with $d^i_V$ for each $i$, so we obtain a pullback for the twisted complexes. Further the pullback preserves each of the spaces $\ker(S^{i,j})$ and $\ran(S^{i,j})$ as well as their orthocomplements. Hence the pullback is also compatible with the operators $T^{i,j}$ and there is an induced pullback on the spaces $\Upsilon^i$. This in turn implies that all the further operations we construct are compatible with the pullback and hence so are the BGG operators $D^i$. The pullbacks on the spaces $\Upsilon^i$ can be made explicit, they correspond to the natural pullback on tensor fields of appropriate type. In particular, if some Euclidean motion maps a bounded Lipschitz domain to itself then all our constructions are invariant under that motion.

 For more general maps, the compatibility need not hold. The corresponding phenomenon for BGG complexes related to projective differential geometry involves affine transformations instead of Euclidean motions. The failure of compatibility with more general maps was encountered in the context of the isogeometric analysis \cite{arf2021structure} for the Hessian complex. 

 \subsection{Conformal deformation complex}\label{sec:conf-def}

  In the language of Section \ref{sec:struct}, we put $n=3$, $N=2$, $\mathbb V_0=\mathbb R^3$, $\mathbb V_1:={ \mathbb R\oplus \mathfrak o(3)}$, and $\mathbb V_2=\mathbb R^{3*}$, {so in the language of Section \ref{sec:background} we use $\Bbb V=\mathfrak g$ (with a shift of the grading by $1$ for consistency).} The conceptual notation to use here is to write elements of $Z^{i,0}$ as column vectors of $i$-forms $\phi_\ell$, and elements of $Z^{i,2}$ as row vectors of $i$-forms $\omega_\ell$ for $\ell=1,2,3$. Elements of $Z^{i,1}$ in this convention are viewed as $3\times 3$-matrices of $i$-forms of the form ${(\psi,\psi_\ell)}:=\begin{pmatrix} \psi & \psi_3 & -\psi_2\\ -\psi_3 & \psi & \psi_1\\ \psi_2 & -\psi_1 & \psi \end{pmatrix}$. Next, $K^{i,1}$ is defined by acting with a matrix on the positions vector, i.e.
   $$
   K^{i,1}\begin{pmatrix} \psi & \psi_3 & -\psi_2\\ -\psi_3 & \psi & \psi_1\\ \psi_2 & -\psi_1 & \psi \end{pmatrix}=\begin{pmatrix} x^1\psi + x^2\psi_3 -x^3\psi_2\\ -x^1\psi_3+x^2\psi+x^3\psi_1\\ x^1\psi_2-x^2\psi_1+x^3\psi \end{pmatrix}.$$
   As for the conformal Hessian complex, this readily leads to
   $$
   S^{i,1}\begin{pmatrix} \psi & \psi_3 & -\psi_2\\ -\psi_3 & \psi & \psi_1\\ \psi_2 & -\psi_1 & \psi \end{pmatrix}=\begin{pmatrix} dx^1\wedge\psi + dx^2\wedge\psi_3 -dx^3\wedge\psi_2\\ -dx^1\wedge\psi_3+dx^2\wedge\psi+dx^3\wedge\psi_1\\ dx^1\wedge\psi_2-dx^2\wedge\psi_1+dx^3\wedge\psi \end{pmatrix}.
   $$
   The operator $K^{i,2}$ can be  obtained from a Lie bracket with a position vector, see Section \ref{sec:background}, but one can simply start with an explicit formula:
   $$
   K^{i,2}(\omega_1,\omega_2,\omega_3):=\begin{pmatrix}\textstyle\sum_\ell x^\ell\omega_\ell & x^1\omega_2 -x^2\omega_1& x^1\omega_{3}-x^3\omega_1\\ x^2\omega_1-x^1\omega_2 &  \textstyle\sum_\ell x^\ell\omega_\ell & x^2\omega_3-x^3\omega_2 \\ x^3\omega_1-x^1\omega_3 & x^3\omega_2-x^2\omega_3 & \textstyle\sum_\ell x^\ell\omega_\ell \end{pmatrix}, 
   $$
   and this readily gives
   $$
S^{i,2}(\omega_1,\omega_2,\omega_3):=\begin{pmatrix}\textstyle\sum_\ell dx^\ell\wedge\omega_\ell & dx^1\wedge\omega_2-dx^2\wedge\omega_1 & dx^1\wedge\omega_{3}-dx^3\wedge\omega_1\\ dx^2\wedge\omega_1-dx^1\wedge\omega_2 &  \textstyle\sum_\ell dx^\ell\wedge\omega_\ell & dx^2\wedge\omega_3-dx^3\wedge\omega_2 \\ dx^3\wedge\omega_1-dx^1\wedge\omega_3 & dx^3\wedge\omega_2-dx^2\wedge\omega_3 & \textstyle\sum_\ell dx^\ell\wedge\omega_\ell \end{pmatrix}.
$$
Now as for the conformal Hessian complex, there is just one instance of \eqref{SKKS} that needs to be {checked}, namely $K^{i+1,1}\circ S^{i,2}=S^{i,1}\circ K^{i,2}$. This is verified by a simple direct computation: The first rows of $K^{i+1,1}(S^{i,2}(\omega_1,\omega_2,\omega_3))$ and of $S^{i,1}(K^{i,2}(\omega_1,\omega_2,\omega_3))$ are, respectively, given by 
\begin{gather*}
  \textstyle\sum_\ell x^1dx^\ell\wedge\omega_\ell-x^2dx^2\wedge\omega_1+x^2dx^1\wedge\omega_2-x^3dx^3\wedge\omega_1+x^3dx^1\wedge\omega_3\\
  +\textstyle\sum_\ell x^\ell dx^1\wedge\omega_\ell-x^2dx^2\wedge\omega_1+x^1dx^2\wedge\omega_2-x^3dx^3\wedge\omega_1+x^1dx^3\wedge\omega_3, 
\end{gather*}
so these agree. Similarly, one verifies that the other rows agree. Converting to vector proxies, we obtain the following BGG diagram 
 \begin{equation}\label{conformal-3rows}
\begin{tikzcd}
0 \arrow{r}&H^{q_0}\otimes \mathbb{V}\arrow{r}{\grad} &H^{q_0-1}\otimes \mathbb{M}  \arrow{r}{\curl} &H^{q_0-2}\otimes \mathbb{M} \arrow{r}{\div} & H^{q_0-3}\otimes \mathbb{V} \arrow{r}{} & 0\\
0 \arrow{r} &H^{q_1}\otimes (\mathbb{R}\oplus\mathbb{V})\arrow{r}{\grad}  \arrow[u, "{K^{0,1}}"] \arrow[ur, "S^{0, 1}"] &H^{q_1-1}\otimes (\mathbb{V}\oplus \mathbb{M})  \arrow{r}{\curl} \arrow[u, "{K^{1,1}}"] \arrow[ur, "S^{1, 1}"]&H^{q_1-2}\otimes (\mathbb{V}\oplus \mathbb{M}) \arrow{r}{\div}\arrow[u, "{K^{2,1}}"]\arrow[ur, "S^{2, 1}"] & H^{q_1-3}\otimes (\mathbb{R}\oplus \mathbb{V}) \arrow[u, "{K^{3,1}}"]\arrow{r}{} & 0\\
0 \arrow{r}&H^{q_2}\otimes \mathbb{V}\arrow{r}{\grad} \arrow[u, "{K^{0,2}}"]\arrow[ur, "S^{0, 2}"]&H^{q_2-1}\otimes \mathbb{M}  \arrow{r}{\curl} \arrow[u, "{K^{1,2}}"]\arrow[ur, "S^{1, 2}"]&H^{q_2-2}\otimes \mathbb{M} \arrow{r}{\div}\arrow[u, "{K^{2,2}}"]\arrow[ur, "S^{2, 2}"] & H^{q_2-3}\otimes \mathbb{V} \arrow[u, "{K^{3,2}}"]\arrow{r}{} & 0.
 \end{tikzcd}
\end{equation}
 The $S$ and $K$ operators can  alternatively be obtained from the   diagram \cite[(32)]{arnold2021complexes} where we have maps between scalar and vector-valued de Rham complexes.
Then the corresponding operators in \eqref{conformal-3rows} is a combination of them.  Specifically, $S^{0, 1}=(\iota, -\mskw)^{t}$, i.e., $S^{0, 1}(w, v)=wI-\mskw v$. Similarly, $S^{1, 1}=(-\mskw, \mathcal{S})^{t}$, $S^{2, 1}=(I, 2\vskw)^{t}$, $S^{0, 2}=(I, -\mskw)$, $S^{1, 2}=(2\vskw,  \mathcal{S})$, $S^{2, 2}=(\tr, 2\vskw)$. The $K$ operators are defined by $K^{1, i}=(\otimes x, \wedge x)^{t}$, $K^{2, i}=(\cdot x, \wedge x)$.

  To understand the form of the resulting BGG complex, we observe that the definitions easily imply that the linear maps $\partial^{i,j}$ that underlie $S^{0,1}$, $S^{0,2}$ and $S^{1,2}$ are injective while those for $S^{1,1}$, $S^{2,1}$ and $S^{2,2}$ are surjective. Since $\partial^{1,1}\circ\partial^{0,2}=0$, the dimensions of the spaces in question imply that $\ker(\partial^{1,1})=\ran(\partial^{0,2})$ so the corresponding statement for the $S$-operators follows. In the same way, we conclude that $\ker(S^{2,1})=\ran(S^{1,2})$. This implies that non-zero harmonic parts {$\Upsilon^{i,j}$ in \eqref{algebraic-hodge}} only occur for $(i,j)=(0,0)$, $(1,0)$, $(2,2)$, and $(3,2)$ and {using formula \eqref{D-matrix} from the appendix}, we conclude that the resulting BGG complex is
 \begin{equation}
\begin{tikzcd}
0 \arrow{r}&H^{q}\otimes \mathbb{V}\arrow{r}{\dev\deff} &H^{q-1}\otimes (\mathbb{S}\cap \mathbb{T})  \arrow{r}{\cot} &H^{q-4}\otimes (\mathbb{S}\cap \mathbb{T}) \arrow{r}{\div} & H^{q-5}\otimes \mathbb{V} \arrow{r}{} & 0.
 \end{tikzcd}
\end{equation}
 Here $\dev\deff=\dev\sym\grad$ is the symmetric trace-free part of the gradient, and
 $\cot:={\curl  \mathcal{S}^{-1}\curl  \mathcal{S}^{-1}\curl}$ leads to the linearized Cotton-York tensor with modified trace.  Compatibility of the construction with Euclidean motions can be obtained similarly as for the conformal Hessian complex. However, this time the action of a motion mixes components in all three rows. We do not go into details of this here.

\subsection{Higher order generalizations of the Hessian complex}\label{sec:higher-Hess}
The next example deals with an arbitrary number of rows. It should be mentioned that the setup here is unusually simple, in general it is very hard to treat complexes in higher orders simultaneously. Indeed, the Hessian complex is the simplest example of a BGG complex (beyond the de Rham complex). For simplicity, we again restrict to the case $n=3$, but we take an arbitrary number $N\in\mathbb N$ and define $\mathbb V_j$ to be the symmetric power $\Symm^j\mathbb R^{3*}$ for $j=0,\dots, N$, so in particular $\mathbb V_0=\mathbb R$ and $\mathbb V_1=\mathbb R^{3*}$ (the definition for symmetric powers can be found in, e.g., \cite{fulton2013representation}). {Elements of $Z^{i,0}$ will be viewed as $i$-forms, while for $i>0$, we view an} element of $Z^{i,j}$ as a family $\phi_{k_1\dots k_j}$ of differential forms of degree $i$ indexed by $j$-numbers $k_1,\dots,k_j\in\{1,2,3\}$ which are completely symmetric, i.e.~$\phi_{k_1\dots k_j}=\phi_{k_{\sigma_1}\dots k_{\sigma_j}}$ for any permutation $\sigma$ of $j$ elements.  { In other words, the $(i, j)$-th space in the BGG diagram is $\Lambda^{j} \mathbb R^{3*}\otimes \Symm^{i}\mathbb R^{3*}$ (here we count $i$ and $j$ from zero).}

  In this language, we define $K^{i,j}$ for $j\geq 1$ by
    $$
    K^{i,j}(\phi)_{k_1,\dots,k_{j-1}}:=\textstyle\sum_\ell x^\ell\phi_{\ell k_1\dots k_{j-1}},
    $$
    so symmetry of the family $\phi_{k_1\dots k_j}$ implies symmetry of $K^{i,j}(\phi)_{k_1\dots k_{j-1}}$. As before, this immediately implies that
    $$
    S^{i,j}(\phi)_{k_1,\dots,k_{j-1}}:=\textstyle\sum_\ell dx^\ell\wedge\phi_{\ell k_1\dots k_{j-1}}. 
    $$
    Hence we conclude that $S^{i,j}$ is indeed induced by composition with a linear map $\partial^{i,j}:\Lambda^i\mathbb R^{3*}\otimes \Symm^j\mathbb R^{3*}\to \Lambda^{i+1}\mathbb R^{3*}\otimes \Symm^{j-1}\mathbb R^{3*}$. To verify \eqref{SKKS} we observe that for $j\geq 2$, we get 
    $$
    S^{i,j-1}(K^{i,j}(\phi))_{k_1\dots k_{j-2}}=\textstyle\sum_{\ell,m}x^\ell dx^m\wedge\phi_{m\ell k_1\dots k_{j-2}}, 
    $$
    while $K^{i+1,j-1}(S^{i,j}(\phi))_{k_1\dots k_{j-2}}$ is given by the same expression with $\phi_{\ell m k_1\dots k_{j-2}}$ instead. Hence \eqref{SKKS} holds by symmetry of the indices of $\phi$. Thus we have completed all verifications needed to get the machinery going, and it remains to understand what the resulting BGG complex looks like.
    
We first observe that $\partial^{0,1}:\mathbb R^{3*}\to\mathbb R^{3*}$ is the identity, so in the case $q_j=q-j$, which is relevant for step 2, $S^{0,1}$ is the identity, too. In the next ``diagonal'' we get 
    $$
    \Symm^2\mathbb R^{3*}\overset{\partial^{0,2}}{\longrightarrow} \mathbb R^{3*}\otimes\mathbb R^{3*}
    \overset{\partial^{1,1}}{\longrightarrow} \Lambda^2\mathbb R^{3*},
    $$
    and $\partial^{0,2}$ is just the inclusion, while $\partial^{1,1}$ is the alternation. So we conclude that $S^{0,2}$ is injective, $\ran(S^{0,2})=\ker(S^{1,1})$ and $S^{1,1}$ is surjective. The next diagonals look uniform for {$j=3,\dots,N$}. One always gets
    $$
    \Symm^j\mathbb R^{3*}\overset{\partial^{0,j}}{\longrightarrow} \mathbb R^{3*}\otimes \Symm^{j-1}\mathbb R^{3*}
    \overset{\partial^{1,j-1}}{\longrightarrow} \Lambda^2\mathbb R^{3*}\otimes \Symm^{j-2}\mathbb R^{3*}
    \overset{\partial^{2,j-2}}{\longrightarrow} \Lambda^3\mathbb R^{3*}\otimes \Symm^{j-3}\mathbb R^{3*}. 
    $$
    Here $\partial^{0,j}$ is the inclusion while the following maps are induced by alternations {in the first $i+1$ arguments}.   The first space in this sequence can be interpreted as homogeneous polynomials of degree $j$ on $\Bbb R^n$. Next, we have one-forms, whose coefficients are homogeneous polynomials of degree $j-1$ and then two-forms with coefficients that are homogenous of degree $j-2$. The last space consists of three-forms, whose coefficients are homogeneous polynomials of degree $j-3$. {Moreover, the maps $\partial^{i,j}$ are obtained by restricting the exterior derivative to polynomial forms, whence} this is sometimes called a \textit{polynomial de Rham complex} at degree $j$. It  is well known (see \cite{Arnold.D;Falk.R;Winther.R.2006a} or Section 6.1 of \cite{Seiler:Involution}) that this is exact for $j\geq 3$ with $\partial^{0,j}$ injective and $\partial^{2,j-2}$ surjective. 
    Hence for $j=3,\dots,N$,  $S^{0,j}$ is injective, $\ker(S^{i,j-i})=\ran(S^{i-1,j-i+1})$ for $i=1,2$ and $S^{2,j-2}$ is surjective. Indeed, this also deals with the remaining diagonals, namely
    \begin{gather*}
      \mathbb R^{3*}\otimes \Symm^N\mathbb R^{3*} \overset{\partial^{1,N}}{\longrightarrow} \Lambda^2\mathbb R^{3*}\otimes \Symm^{N-1}\mathbb R^{3*} \overset{\partial^{2,N-1}}{\longrightarrow} \Lambda^3\mathbb R^{3*}\otimes \Symm^{N-2}\mathbb R^{3*} \\
     \Lambda^2\mathbb R^{3*}\otimes \Symm^N\mathbb R^{3*} \overset{\partial^{2,N}}{\longrightarrow} \Lambda^3\mathbb R^{3*}\otimes \Symm^{N-1}\mathbb R^{3*}. 
    \end{gather*} 
    They both have nontrivial cohomology only in {the first position}, and this cohomology is {$\Symm^{N+1}\mathbb R^{3*}$} for the first sequence. For the second sequence, the cohomology is just the kernel $W$ of the alternation over the first three indices in $\Lambda^2\mathbb R^{3*}\otimes \Symm^N\mathbb R^{3*}$. Since we obviously have cohomology $\mathbb R$ in degree $(0,0)$ and $\Lambda^3\mathbb R^{3*}\otimes \Symm^N\mathbb R^{3*}$ in degree $(3,N)$, we conclude that the BGG sequence gets the form
      \begin{equation}
\begin{tikzcd}
0 \arrow{r}{} &H^{q}\arrow{r}{D^0} &H^{q-N-1}\otimes (\Symm^{N+1}\mathbb R^{3*}) \arrow{r}{D^1} &H^{q-N-2}\otimes W \arrow{r}{D^2} & H^{q-N-3}\otimes \Symm^N\mathbb R^{3*}\arrow{r}{} & 0.
 \end{tikzcd}
   \end{equation}
  Here $D^0$ is just the $(N+1)$-fold derivative of a function, while the other two operators are of first order.

\section{Algebraic and geometric background}\label{sec:background}

Here we want to briefly sketch the relation of our constructions and of the examples discussed in Section \ref{sec:examples}  to the notions of BGG resolutions in geometry as developed in \cite{vcap2001bernstein}. The motivation for the construction and the basic input for the examples comes from representation theory. Representation theory in particular provides the maps $\partial^{i,j}$ used in Section \ref{sec:struct} and hence the operators $S^{i,j}$. It also provides similar maps inducing the operators ${T^{i,j}}$, but these are not {needed} here. The geometric version works in the general setting of parabolic subalgebras in semi-simple Lie algebras, but we focus on two special cases here. In both these cases, one deals with a real simple Lie algebra $\frak g$, which admits a Lie algebra grading of the form $\frak g=\frak g_{-1}\oplus\frak g_0\oplus\frak g_1$. This means that $\frak g_{-1}$ and $\frak g_1$ are abelian subalgebras (so they are just vector spaces) and $\frak g_0$ is a Lie subalgebra of $\frak g$. Via the bracket, $\frak g_0$ acts on $\frak g_{-1}$ and $\frak g_1$, and general results imply that these two representations are always dual to each other.

The first case originates in projective differential geometry and relates for example to the elasticity complex, see Section \ref{sec:Cosserat} below, and to the higher order analogs of the Hessian complex discussed in Section \ref{sec:higher-Hess}. Here $\frak g=\frak{sl}(n+1,\Bbb R)$, {the Lie algebra of trace-free $(n+1)\times(n+1)$-matrices}, $\frak g_0=\frak{gl}(n,\Bbb R)$, $\frak g_{-1}\cong\Bbb R^n$ and $\frak g_1\cong\Bbb R^{n*}$.
  The decomposition of $\frak g$ can be simply realized by decomposing {(trace-free)} matrices of size $(n+1)\x (n+1)$ into blocks of sizes $1$ and $n$ and viewing elements of $\Bbb R^n$ as column vectors and elements of $\Bbb R^{n*}$ as row vectors. Explicitly, denoting by $\Bbb I$ the $n\x n$ unit matrix, the matrix corresponding to $v\in\Bbb R^n=\frak g_{-1}$,  $B\in\frak{gl}(n,\Bbb R)=\frak g_0$, and $\lambda\in\Bbb R^{n*}=\frak g_1$ is given by
  \begin{equation}\label{eq:proj-matrix}
  \begin{pmatrix} -\tfrac{1}{n+1}\tr(B) & \lambda\\ v & B-\tfrac{1}{n+1}\tr(B)\Bbb   I \end{pmatrix}.
  \end{equation}
  {Observe that this matrix is always trace-free.} The non-obvious choice for the identification of the block diagonal part is motivated by the fact that
  $$
  \left[\begin{pmatrix}  -\tfrac{1}{n+1}\tr(B) & 0 \\ 0 & B-\tfrac{1}{n+1}\tr(B)\Bbb   I \end{pmatrix}, \begin{pmatrix} 0 & \lambda \\ v & 0 \end{pmatrix}\right]=\begin{pmatrix} 0 & -\lambda B \\ Bv & 0 \end{pmatrix}.
  $$
  This shows that the action of $\frak g_0$ on $\frak g_{\pm 1}$ via the commutator (which is the Lie bracket in $\frak g$) coincides with the standard action on $\Bbb R^n$ and $\Bbb R^{n*}$.  

  The second case relates to conformal geometry and is the basis for the examples discussed in Sections \ref{sec:conf-Hess} and \ref{sec:conf-def}. Here $\frak g=\frak{so}(n+1,1)$, the Lie algebra of linear maps that are skew symmetric with respect to a Lorentzian inner product $b$ on $\Bbb R^{n+2}$ and $\frak g_0=\frak{co}(n)$, the Lie algebra of $n\x n$-matrices spanned by skew symmetric matrices and multiples of the identity. As above, $\frak g_{-1}\cong\Bbb R^n$ and $\frak g_1\cong\Bbb R^{n*}$ are just the standard representation of $\frak g_0$ and its dual. {To obtain an explicit matrix representation, one uses a basis $v_0,\dots, v_{n+1}$ of $\Bbb R^{n+2}$ such that $b(v_0,v_{n+1})=b(v_{n+1},v_0)=1$ and $b(v_i,v_j)=\delta_{ij}$ for $i,j\in\{1,\dots,n\}$ and all other values of $b$ on basis vectors are zero. In such a basis $\frak g$ can be realized as all matrices of size $(n+2)\times (n+2)$ which have a block form
    $$
\begin{pmatrix} a & Z & 0\\ X & A & -Z^t\\ 0 & -X^t & -a\end{pmatrix}
  $$
  with block sizes $(1,n,1)$ and $A^t=-A$. Here the grading comes from the block form, i.e.\ $X$ corresponds to $\frak g_{-1}$, $a$ and $A$ correspond to $\frak g_0$ and $Z$ corresponds to $\frak g_1$, see Section 1.6.3 of \cite{parabolbook} for more details.}

Now the starting point for the BGG construction is a representation $\Bbb V$ of $\frak g$, which can be assumed to be irreducible. By restriction, $\Bbb V$ becomes a representation of $\frak g_0$, but this is not irreducible anymore. One also obtains actions of $\frak g_{\pm 1}$ on $\Bbb V$ and one can decompose $\Bbb V$ into a direct sum $\Bbb V=\Bbb V_0\oplus\dots\oplus\Bbb V_N$ for some $N\geq 0$ in such a way that the action of $\frak g_1$ maps each $\Bbb V_i$ to $\Bbb V_{i+1}$ while the action of $\frak g_{-1}$ maps each $\Bbb V_i$ to $\Bbb V_{i-1}$. (Here we use the convention that $\Bbb V_i=\{0\}$ if $i<0$ or $i>N$.)  The Lie algebra $\frak g$ is semi-simple and the resulting algebra $\frak g_0$ turns out to be reductive. Hence representations of both algebras can be understood in terms of so-called highest weights. In this language, there is also a scheme for describing the possible first operators in a BGG sequence and the choice of $\Bbb V$ that is needed to produce a chosen operator from that list, see \cite{BCEG}. This is beyond the scope of the current article, however.  

 For example, corresponding to the block decomposition of $\frak{sl}(n+1,\Bbb R)$ from above, we have to decompose the standard representation $\Bbb R^{n+1}$ of $\frak g$ into the top component of a column vector in $\Bbb R^{n+1}$ and the remaining vector in $\Bbb R^n$. Restricting to the case of trace-free matrices in $\frak g_0$, we see that
 $$
\begin{pmatrix} 0 & \lambda \\ v & B
\end{pmatrix}\begin{pmatrix} a \\ w
\end{pmatrix} = \begin{pmatrix} \lambda(w) \\ Bw+av
\end{pmatrix}
  $$
  This shows that $\Bbb V_0=\Bbb R^n$ and $\Bbb V_1=\Bbb R$ as representations of $\frak{sl}(n,\Bbb R)\subset\frak g_0$, the action $\frak g_{-1}\x \Bbb V_1\to \Bbb V_0$ sends $(v,a)$ to $av$ and the action $\frak g_1\x\Bbb V_0\to\Bbb V_1$ is just the dual pairing $\Bbb R^{n*}\x\Bbb R^n\to\Bbb R$. Similarly, putting $\Bbb V=\Bbb R^{(n+1)*}$, the dual of the standard representation, we get $\Bbb V_0=\Bbb R$ and $\Bbb V_1=\Bbb R^{n*}$ with obvious actions of $\frak g_{\pm 1}$. This choice of $\Bbb V$ leads to the Hessian complex. For the higher order analogs discussed in Section \ref{sec:higher-Hess} one has to use $\Bbb V:=\odot^N\Bbb R^{(n+1)*}$, which accordingly decomposes as indicated there. 

 Returning to a general representation $\Bbb V$, the action of $\frak g_1$ can be interpreted as defining a linear map $\frak g_1\otimes\Bbb V\to\Bbb V$, which maps $\frak g_1\otimes\Bbb V_i$ to $\Bbb V_{i+1}$ for each $i$. It is a purely algebraic fact that this extends to a sequence of linear maps $\Lambda^k\frak g_1\otimes\Bbb V\to\Lambda^{k-1}\frak g_1\otimes \Bbb V$, which also send $\Lambda^k\frak g_1\otimes\Bbb V_i$ to $\Lambda^{k-1}\frak g_1\otimes \mathbb V_{i+1}$.
These are Lie algebra homology differentials, but for historical reasons they are often referred to as the \textit{Kostant codifferential} and denoted by $\partial^*$. They satisfy $\partial^*\circ\partial^*=0$ and hence define a complex, which computes the Lie algebra homology of $\frak g_1$ with coefficients in $\Bbb V$. 

For the action of $\frak g_{-1}$, one uses a different but equivalent encoding.
 Rather than as a map $\frak g_{-1}\x\Bbb V\to \Bbb V$, we view the action as defining a map $\Bbb V\to L(\frak g_{-1},\Bbb V)$. Namely, we send an element $v\in\Bbb V$ to the map $\frak g_{-1}\to \Bbb V$, $X\mapsto X\cdot v$. A general algebraic construction extends this to a map $\partial$ that sends $k$-linear alternating maps $(\frak g_{-1})^k\to\Bbb V$ to $(k+1)$-linear alternating maps $(\frak g_{-1})^{k+1}\to\Bbb V$. This is the \textit{Lie algebra cohomology differential}, which in the case of the abelian Lie algebra $\frak g_{-1}$ is explicitly given by
  $$
\partial \phi(X_0, \cdots, X_k):=\textstyle\sum_{i=0}^k(-1)^{i}X_i\cdot \phi(X_{0}, \cdots, \widehat{X_{i}}, \cdots, X_k),
  $$
with the hat denoting omission and the dot denoting the action of $\frak g_{-1}$ on $\Bbb V$. The construction via a representation of $\frak g$ ensures that this is equivariant for the action of $\frak g_0$ and it is easily seen to satisfy $\partial\circ\partial=0$. Hence it defines a complex that computes the Lie algebra cohomology of $\frak g_{-1}$ with coefficients in $\Bbb V$. Using the duality between $\frak g_{-1}$ and $\frak g_1$, the space of $k$-linear maps from above can again be interpreted as $\Lambda^k\frak g_1\otimes\Bbb V$. Thus we can interpret both $\partial^*$ and $\partial$ as acting on $\Lambda^\bs \frak g_1\otimes \Bbb V$, one lowering and one raising the degree in the exterior algebra, and both preserving the total degree (which is $k+i$ on $\Lambda^k\frak g_1\otimes\Bbb V_i$).

The whole setup was introduced (in the setting of general parabolic subalgebras) in B.\ Kostant's work \cite{Kostant}, where it was next shown that $\partial$ and $\partial^*$ are adjoint with respect to an appropriate inner product on $\Lambda^*\frak g_1\otimes \Bbb V$. Defining the \textit{Kostant Laplacian} $\square:=\partial^*\circ\partial+\partial\circ\partial^*$ one obtains a map that sends each $\Lambda^k\frak g_1\otimes\Bbb V$ to itself. The constructions easily imply that both $\partial$ and $\partial^*$ are $\frak g_0$-equivariant, and one obtains an algebraic Hodge decomposition
\begin{equation}\label{Hodge}
\Lambda^k\frak g_1\otimes\Bbb V=\Cal R(\partial)\oplus\Cal N(\square)\oplus\Cal R(\partial^*)
\end{equation}
into $\frak g_0$-invariant subspaces for each $k$. This has the property that the first two summands add up to $\ker(\partial)$ while the last two summands add up to $\ker(\partial^*)$. This easily implies that $\partial$ restricts to an isomorphism between the $\Cal R(\partial^*)$-component in $\Lambda^k\frak g_1\otimes\Bbb V$ and the $\Cal R(\partial)$-component in $\Lambda^{k+1}\frak g_1\otimes\Bbb V$, and $\partial^*$ restricts to an isomorphism in the opposite direction. The next step in \cite{Kostant} is {showing} that the map $\square$ can be nicely interpreted in terms of representation theory. This leads to an explicit description of $\ker(\square)$ and hence of the Lie algebra (co)homology spaces described above in representation theory terms, which is known as \textit{Kostant's version of the Bott-Borel-Weyl theorem}. 

\smallskip

This setup can be translated to geometry, more precisely to projective differential geometry in dimension $n\geq 2$ in the case that $\frak g=\frak{sl}(n+1,\Bbb R)$ and to conformal differential geometry of dimension $n\geq 3$ for $\frak g=\frak{so}(n+1,1)$. {The point about this is that representations of $\frak{gl}(n,\Bbb R)$ are closely related to natural vector bundles on smooth manifolds of dimension $n$. Similarly, vector representations of $\frak{co}(n)$ are closely related to natural vector bundles on $n$-dimensional manifolds endowed with a Riemannian metric. (Indeed, conformal equivalence class of such metrics suffices.) Formally, one needs representations of the groups $GL(n,\Bbb R)$ respectively $CO(n)$ here, which needs some care in the case of general representations $\Bbb V$. In both cases, the standard representation $\Bbb R^n\cong\frak g_{-1}$ corresponds to the tangent bundle $TM$ and its dual $\Bbb R^{n*}\cong\frak g_1$ corresponds to the cotangent bundle $T^*M$. For all examples discussed in this article, the representations in question can be obtained from these two representations via constructions that can also be applied to vector bundles, which leads to appropriate interpretations.} Moreover, $\frak g_0$-equivariant linear maps between representations induce natural bundle maps between the corresponding natural vector bundles. 

We only describe the next steps in the simpler setting of a manifold $M$ endowed with a linear connection $\nabla$ on $TM$ respectively with a Riemannian metric for which $\nabla$ is the Levi-Civita connection. Then there is an induced connection on any of the natural vector bundles described above and we will denote all these connections by $\nabla$. Given a representation $\Bbb V$ of $\frak g$ as above, we can simply view it as a representation of $\frak g_0$ and then consider the corresponding natural bundle on some manifold $M$, which we denote by $\Cal VM\to M$. (To deal with the setting of projective or conformal differential geometry, a different interpretation is needed, but we don't go into these aspects here.) The representations $\Lambda^k\frak g_1\otimes\Bbb V$ then correspond to the bundles $\Lambda^kT^*M\otimes\Cal VM$ whose sections are $\mathcal VM$-valued $k$-forms. Hence the maps $\partial$ and $\partial^*$ from above give rise to natural bundle maps on these bundles of differential forms. On the other hand, the connection $\nabla$ on $\Cal VM$ can be coupled to the exterior derivative on differential forms to obtain the \textit{covariant exterior derivative} on $\Cal VM$-valued differential forms.

Motivated by ideas from projective and conformal differential geometry (which generalize to so-called parabolic geometries) one can use the connection $\nabla$, the bundle map $\partial$ and some components of the curvature of $\nabla$ to define a new connection $\nabla^{\Cal V}$ on the bundle $\Cal VM$. This can be arranged in such a way that $\nabla^{\Cal V}$ is flat, even if the given connection $\nabla$ on $TM$ is only projectively flat respectively the given Riemannian metric is only conformally flat. {The connection $\nabla^{\Cal V}$} again can be coupled to the exterior derivative to obtain an operation on $\Cal VM$-valued forms that defines a complex generalizing the de Rham complex. Together with the bundle maps induced by $\partial^*$, this operation can be used to obtain a BGG complex of higher order natural differential operators from this twisted de Rham complex.

\smallskip

Things simplify considerably in the setting we use in this article. We restrict to the case of the flat connection, respectively the flat metric, on $\Bbb R^n$ or an open subset $U\subset\Bbb R^n$ and do not require projective or conformal invariance. This allows us to ignore the maps $\partial^*$ and focus on the maps $\partial$ (which can be modified as long as the kernel and the image remain unchanged). The map $\partial^{i,j}$ used in Section \ref{sec:struct} is then the restriction of $\partial$ to $\Lambda^i\Bbb R^{n*}\otimes \mathbb V_j$, which by construction has values in $\Lambda^{i+1}\Bbb R^{n*}\otimes \mathbb V_{j-1}$. The algebraic Hodge decomposition \eqref{Hodge} gets replaced by \eqref{algebraic-hodge} and we can define the operators $T$ as partial inverses of $S$. Since the tangent bundle of $\Bbb R^n$ is trivialized by sections that are parallel for $\nabla$, the same holds for any of the natural bundles described above. Hence $\Cal VU$ can be identified with $U\x\Bbb V$ and hence the space of $\Cal VU$-valued forms can be identified with $\Omega^k(U)\otimes\Bbb V$ {(and further with functions $U\to\Lambda^k\Bbb R^{n*}\otimes\Bbb V$)}. Under this identification, the covariant exterior derivative associated to $\nabla$ simply becomes $d\otimes\operatorname{id}_{\Bbb V}$. Further, as a representation of $\frak g_0$, we have $\Bbb V=\Bbb V_0\oplus\dots\oplus\Bbb V_N$, which gives the splitting into rows that we use. Since $\nabla$ has trivial curvatures, we only need the operators $S$ to construct $\nabla^{\Cal V}$ and the corresponding covariant exterior derivative $d_V$. 

We use the convention that $\nabla^{\Cal V}=\nabla-S$ and then $d_V=d-S$. The fact that $\nabla^{\Cal V}$ is flat is then equivalent to $d_V\circ d_V=0$ and hence to \eqref{DSSD} and \eqref{SS}. Flatness of the connection $\nabla^{\Cal V}$ implies that the bundle $\Cal V\Bbb R^n$ can be globally trivialized by sections that are parallel for $\nabla^{\Cal V}$. This trivialization is the basis for the isomorphism in Theorem \ref{thm:input-twisted}. In the special case of $\Bbb R^n$, there is a neat description of this isomorphism using the operators we call $K$ in \eqref{SKKS}. This description has no analog in the general theory and is not projectively or conformally invariant.
 
\section{Conformal Korn inequalities in 2D}

In this section, we use $\|u\|:=(\int_{U}u^{2}\,dx)^{\frac{1}{2}}$ and $\|u\|_{1}:=(\|u\|^{2}+\|\nabla u\|^{2})^{\frac{1}{2}}$ to denote the $L^{2}$ and $H^{1}$ norms, respectively, {and similarly for higher Sobolev norms.} The (second) Korn inequality refers to
$$
\|u\|_{1}\leq C\|\deff u\|, 
$$
{ which holds for all $u\in H^{1}(\Omega)\otimes \mathbb{R}^{n}$ such that $\int_{\Omega}u\cdot q\, dx=0, ~\forall q\in \mathrm{RM}$.} Here $C$ is a constant not depending on $u$, and $\mathrm{RM}:=\ker(\deff)$ is the finite dimensional space of Killing vector fields, i.e.\ infinitesimal rigid body motions  (see, e.g., \cite{ciarlet2013linear}). The orthogonality condition against $\mathrm{RM}$ is to exclude the kernel of $\deff$.  
There are other ways to fix the kernel. For example, one has the first Korn inequality with boundary conditions.

The following generalization of the Korn inequality, sometimes referred to as the conformal (trace-free) Korn inequality, holds
\begin{equation}\label{conformal-korn-ineq}
\|u\|_{1}\leq C\|\dev\deff u\|, \quad\forall u\in H^{1}(\Omega)\otimes \mathbb{R}^{n}, ~n\geq 3, ~ \int_{\Omega}u\cdot q=0, ~\forall q\in \ker(\dev\deff),
\end{equation}
where $\ker(\dev\deff)$ is the space of conformal Killing vectors,  i.e.\ infinitesimal conformal motions. The trace-free Korn inequality has various applications in, e.g.,  general relativity \cite{dain2006generalized} and continuum theory with microstructures \cite{jeong2010existence}.  The trace-free Korn inequality only holds in $n$D for $n\geq 3$. The failure of this inequality in 2D is discussed in \cite{dain2006generalized}. 

The framework in \cite{arnold2021complexes} provides a systematic cohomological approach for establishing Poincar\'e type inequalities. The key  is that the differential operators have closed range, which is further because the cohomology has finite dimension. In this framework, the trace-free Korn inequality \eqref{conformal-korn-ineq} holds in 3D because the conformal deformation complex has finite dimensional cohomology on Lipschitz domains, {see  \cite{arnold2021complexes}}. {Of course, this follows from the results in Section \ref{sec:conf-def}. Different from the three-row diagram \eqref{conformal-3rows} with de Rham complexes there,} \cite{arnold2021complexes} used either the first and the second rows or the second and the third rows of the following diagram
\begin{equation}\label{diagram-1}
\begin{tikzcd}
0 \arrow{r}{}&H^{q}\otimes \mathbb{V}\arrow{r}{\dev\grad} & H^{q-1}\otimes \mathbb{T}  \arrow{r}{\sym\curl} & H^{q-2}\otimes \mathbb{S} \arrow{r}{\div\div} & H^{q-4} \arrow{r}{} &  0\\
0 \arrow{r}{}&H^{q-1}\otimes \mathbb{V}\arrow{r}{\deff}\arrow[ur, "-\mskw"]  & H^{q-2}\otimes \mathbb{S}  \arrow{r}{\inc}\arrow[ur, "\mathcal{S}"]  & H^{q-4}\otimes \mathbb{S} \arrow{r}{\div}\arrow[ur, "\tr"]  & H^{q-5}\otimes \mathbb{V} \arrow{r}{} &  0\\
0 \arrow{r}{}&H^{q-2}\arrow{r}{\hess}\arrow[ur, "\iota"] & H^{q-4}\otimes \mathbb{S}  \arrow{r}{\curl} \arrow[ur, "\mathcal{S}"]& H^{q-5}\otimes \mathbb{T} \arrow{r}{\div}\arrow[ur, "2\vskw"] & H^{q-6}\otimes \mathbb{V} \arrow{r}{} &  0.
\end{tikzcd}
\end{equation}
 It is also observed in \cite{arnold2021complexes} that the two dimensional version of \eqref{diagram-1}, i.e., 
\begin{equation}\label{diagram-2D-reduction}
\begin{tikzcd}
0 \arrow{r}{} &H^{q}\otimes \mathbb{V} \arrow{r}{{\deff}} &H^{q-1}\otimes \mathbb{S} \arrow{r}{\rot\rot} & H^{q-3}\arrow{r} & 0\\
0 \arrow{r}{}&H^{q-1} \arrow{r}{{\hess}}\arrow[ur, "\iota"] &H^{q-3}\otimes \mathbb{S} \arrow{r}{\rot} \arrow[ur, "\tr"]& H^{q-4}\otimes \mathbb{V}\arrow{r} & 0,
 \end{tikzcd} 
\end{equation}
does not satisfy the injectivity/surjectivity conditions {of}  \cite{arnold2021complexes}. This is consistent with the fact that the trace-free Korn inequality does not hold in two space dimensions. 

Based on the construction in this paper, we now fix the trace-free Korn inequality in two space dimensions by adding terms involving third order operators. {In the language of Section \ref{sec:background}, the conformal deformation complex can be derived in any dimension $n\geq 3$ starting from $\mathfrak g=\mathfrak{so}(n+1,1)$ and the adjoint representation $\mathbb V=\mathfrak g$. We use the straightforward analog of this construction obtained from $\mathfrak g=\mathfrak{so}(3,1)$ and $\mathbb V=\mathfrak g$. We know from Section \ref{sec:background} that $\mathfrak g_0=\mathfrak{co}(n)$, so for $n=2$, this just gives $\mathbb C$ viewed as a subspace of $M_2(\Bbb R)$. Defining the operators $K^{i,j}$ and $S^{i,j}$ as in higher dimensions, one immediately concludes that  $S^{0,1} $and $S^{0,2}$ are injective while $S^{1,1}$ and $S^{1,2}$ are surjective. Consequently, in degree one the ``cohomology'' $\Upsilon^1=\ran(S^2)^\perp\cap \ker(S^1)$ has two components which both are functions with values in a two-dimensional space. One component sits in the top row and hence gives rise to a first order operator, which can be most easily interpreted as the Cauchy-Riemann operator. The other component sits in the bottom row and thus determines a third order operator. In degree two, the cohomology sits in the bottom row, so one obtains operators of order 3 and 1. Passing to vector proxies the resulting diagram looks as follows.}
 \begin{equation}\label{conformal-3rows-2D}
\begin{tikzcd}
0 \arrow{r}&H^{q}\otimes \mathbb{V}\arrow{r}{\grad} &H^{q-1}\otimes \mathbb{M}  \arrow{r}{\rot}\arrow[dl, "T^{1, 0}", shift left=1] &H^{q-2}\otimes \mathbb{V} \arrow{r}{}\arrow[dl, "T^{2, 0}", shift left=1] &  0\\
0 \arrow{r} &H^{q-1}\otimes (\mathbb{R}\oplus\mathbb{R})\arrow{r}{\grad} \arrow[ur, "S^{0, 1}"]&H^{q-2}\otimes (\mathbb{V}\oplus \mathbb{V})  \arrow{r}{\rot}\arrow[dl, "T^{1, 1}", shift left=1]\arrow[ur, "S^{1, 1}"]&H^{q-3}\otimes (\mathbb{R}\oplus \mathbb{R}) \arrow{r}{}\arrow[dl, "T^{2, 1}", shift left=1] &  0\\
0 \arrow{r}&H^{q-2}\otimes \mathbb{V}\arrow{r}{\grad} \arrow[ur, "S^{0, 2}"]&H^{q-3}\otimes \mathbb{M} \arrow{r}{\rot}   \arrow[ur, "S^{1, 2}"]&H^{q-4}\otimes \mathbb{V}   \arrow{r}{} & 0.
 \end{tikzcd}
\end{equation}
Here {$\rot(u_{1}, u_{2})=-\partial_{2}u_{1}+\partial_{1}u_{2}$,} $S^{0, 1}=(\iota, \mskw)^{t}$, $S^{1, 1}=(-J, I)^{t}$, $S^{0, 2}=(I, J)$, and $S^{1, 2}=(-2\sskw, -\tr)$, with 
$$
{ J \left (
\begin{array}{c}
u_{1}\\
u_{2}
\end{array}
\right ){:=
\left (
\begin{array}{cc}
0 & -1\\
1& 0
\end{array}
\right )
\left (
\begin{array}{c}
u_{1}\\
u_{2}
\end{array}
\right )
 }=\left (
\begin{array}{c}
-u_{2}\\
u_{1}
\end{array}
\right ),} 
$$
$$
 \sskw 
\left (
\begin{array}{cc}
a & b\\
c & d
\end{array}
\right ):=\frac{1}{2}(b-c),
\quad
\mskw  u:=
\left (
\begin{array}{cc}
0 & u\\
-u & 0
\end{array}
\right ).
$$
 These operators can be obtained by \eqref{DKKD} with $K^{i, 1}(u, v):= u\otimes x +Jv\otimes x$, and $K^{i, 2}(u):=(u\cdot x, Ju\cdot x)$. For \eqref{conformal-3rows-2D}, $T^{1, 0}=(\frac{1}{2}\tr, \sskw)$, $T^{2, 0}={\frac{1}{2}}(J, I)$, $T^{1, 1}={\frac{1}{2}}(I, -J)^{t}$, and $T^{2, 1}=({ -\frac{1}{2}\mskw}, -\frac{1}{2}\iota)^{ t}$.
 
From the matrix form of the operators \eqref{D-matrix}, the BGG operators have the form 
\begin{align*}
D^{0}\left ( 
\begin{array}{c}
u\\0\\0
\end{array}
\right )=
\left (
\begin{array}{c}
P_{\ran^{\perp}}d u\\0\\ P_{\ker}(dT)^{2}d u
\end{array}
\right )&=
\left (
\begin{array}{c}
\dev\deff  u\\0\\ 
\dev\deff T^{1, 1}\grad T^{1, 0}\grad u
\end{array}
\right ),
\end{align*}
and
\begin{align*}
D^{1}\left ( 
\begin{array}{c}
u\\0\\v
\end{array}
\right )=
\left (
\begin{array}{c}
0\\0\\ (dT)^{2}du+ d{v}
\end{array}
\right ),
\end{align*}
where 
\begin{align*}
 (dT)^{2}&du+ du=\rot T^{2, 1}\rot T^{2, 0}\rot u+\rot v.
\end{align*}

The output complex is 
\begin{equation}\label{output-2Dconformal}
\begin{tikzcd}
0 \arrow{r}{} &H^{q}\otimes \mathbb{V} \arrow{r}{{{D}^0}} & 
\left (
\begin{array}{c}
H^{q-1}\otimes (\mathbb{S}\cap \mathbb{T})\\
H^{q-3}\otimes (\mathbb{S}\cap \mathbb{T})
\end{array}
\right ) \arrow{r}{{D}^1} & H^{q-4}\otimes\mathbb{V}\arrow{r} & 0,
 \end{tikzcd} 
\end{equation}
where 
$$
{D}^0:=\left(
\begin{array}{c}
\dev\deff \\  
\dev\deff T^{1, 1}\grad T^{1, 0}\grad 
\end{array}
\right ), \quad
{D}^1=\left(
\begin{array}{c}
\rot T^{2, 1}\rot T^{2, 0}\rot  \\  
 \rot
\end{array}
\right )^{t}.
$$

To summarize, we get the following.
\begin{theorem}[Conformal Korn inequality in two space dimensions]\label{thm:2DKorn}
Let $\Omega\subset \mathbb{R}^{2}$ be a bounded Lipschitz domain. { For any $u\in H^{3}\otimes  \mathbb{V}$ and $u\perp \ker (D^{0})$,} there exists a positive constant $C$, such that { 
$$
\|u\|_{3}\leq C(\|\dev\deff  u\|_2+\| \dev\deff T^{1, 1}\grad T^{1, 0}\grad u\|).
$$}
\end{theorem}
{ 
We recall that the key for proving the inequality in Theorem \ref{thm:2DKorn} is to observe that $D^{0}: H^{3}\otimes \mathbb{V}\to H^{2}\otimes (\mathbb{S}\cap \mathbb{T})\oplus H^{1}\otimes (\mathbb{S}\cap \mathbb{T})$ has closed range, as the cohomology at index one has finite dimension. Then $D^{0}$ is a bijection between Banach spaces $\ker(D^{0})^{\perp}\subset H^{3}\otimes \mathbb{V}$ and $\ran(D^{0})$. The desired inequality immediately follows from the Banach closed range theorem. See \cite[(14)]{arnold2021complexes} for more details.   The orthogonality in Theorem \ref{thm:2DKorn}, $u\perp \ker (D^{0})$, is naturally given with respect to the $H^{3}$ inner product. Nevertheless, as $\ker(D^{0})$ has finite dimension, this condition can be replaced by the $L^{2}$ orthogonality, for example. 

A  more explicit form of the second term $\dev\deff T^{1, 1}\grad T^{1, 0}\grad u$ has the following form. If $u=(u_{1}, u_{2})$, then
$$
\dev\deff T^{1, 1}\grad T^{1, 0}\grad u=
\left ( 
\begin{array}{cc}
a & b\\
b & -a
\end{array}
\right ),
$$
where 
$$
a=\frac{1}{8}(\partial_{1}^{3}u_{1}-\partial_{2}^{3}u_{2}+3\partial_{1}^{2}\partial_{2}u_{2}-3\partial_{1}\partial_{2}^{2}u_{1}),
$$
and
$$
b=\frac{1}{8}(-\partial_{1}^{3}u_{2}-\partial_{2}^{3}u_{1}+3\partial_{1}^{2}\partial_{2}u_{1}+3\partial_{1}\partial_{2}^{2}u_{2}).
$$
}

\begin{remark}
In the classical approaches for the Korn inequalities, the proof for results in $L^{p}$ with $p\neq 2$ are more difficult than the case of $p=2$. {With small modifications,} the complex-based argument in \cite{arnold2021complexes} and in this paper leads to a unified proof for inequalities in a broad range of function spaces, i.e., those spaces in  \cite{costabel2010bogovskiui} for the de Rham complexes. 
 It is known that the Korn and conformal Korn inequalities do not hold for  $W^{1, p}$ with $p=1$ or $p=\infty$ \cite[Chapter 7]{demengel2012functional}, \cite{breit2017trace,breit2012sharp}. This is consistent with the fact that the results in \cite{costabel2010bogovskiui} exclude these two cases.
\end{remark}

{As indicated above, the main difference between the 2D complex and conformal deformation complexes in higher dimensions lies in the location of (Lie algebraic) cohomology, which affects the order of the operators in the resulting BGG sequences. This also occurs in dimensions $n\geq 3$, where it reflects the fact that the fundamental conformal invariants change.} While in dimension $3$, this fundamental invariant is the Cotton-York tensor that depends on third derivatives of a metric in the conformal class, the fundamental invariant for $n\geq 4$ is the Weyl curvature, which only depends on second derivatives  of such a metric. {The special form of the complex} in dimension $2$ also has a geometric background. This is not related to conformal geometry but to a refinement. Indeed, the isomorphism $\mathfrak{co}(2)=\mathbb C$ observed above implies that in dimension $2$, conformal geometry reduces to complex analysis, which also explains the occurrence of the Cauchy-Riemann operator in our complex. To obtain a two dimensional analog of higher-dimensional conformal geometry, one has to add to a conformal structure an additional ingredient called a M\"obius structure. This is given by a symmetric tensor of rank two, which plays the role of an abstract version of the Schouten tensor (a trace-adjustment of the Ricci curvature) in higher dimensions, see \cite{Calderbank:Mobius} for details. The third order part in $D^{0}$ describes the infinitesimal change of that tensor {under a diffeomorphism}, while the Cauchy-Riemann operator describes the infinitesimal change of the conformal structure. In analogy to higher dimensional conformal geometry, an M\"obius-structure always has a finite-dimensional automorphism group with dimension bounded above by $\dim(\mathfrak{so}(3,1))=6$.

\section{Continuum with  microstructures}\label{sec:Cosserat}

 In this section, we establish a cohomological approach for the linear Cosserat model and generalize this model from this new perspective.
 
 {
 \subsection{Overview}

As a basic model, linear elasticity models material with a small deformation and a linear constitutive law, i.e., a linear relation between stress and strain \cite{ciarlet2021mathematical}. Mathematically, the solution of linear elasticity can be obtained as a minimization of an energy functional 
\begin{equation}\label{energy-elasticity}
\|\deff u\|^{2}_{C}-(f, u),
\end{equation}
 where $C$ is a (fourth order) elasticity tensor; the weighted inner product $\|\deff u\|_{C}:=\int \deff u: C: \deff u\, dx$ and $f$ is an external force with $(f, u):=\int f\cdot u\,dx$. 
 
 One immediately recognizes that \eqref{energy-elasticity} is related to the elasticity complex as the quadratic term involves the operator ${D}^{0}$ from the complex. Therefore the linear elasticity equation is an analog of the Poisson equation in the context of the elasticity complex (more precisely, the Euler-Lagrange equation of \eqref{energy-elasticity} corresponds to the Hodge-Laplacian problem of the elasticity complex at index zero).   This simple connection has several consequences. For example, one may introduce the stress as an independent variable and solve a Hodge-Laplacian problem of the elasticity complex at index $n$.
  This leads to the Hellinger-Reissner principle, which is mathematically equivalent to the linear elasticity problem. Other equivalent formulations can be obtained from complexes. For example, the intrinsic formulation \cite{ciarlet2009intrinsic} involves index two of the elasticity complex, while the Hu-Washizu principle \cite{washizu1968variational}, involves the displacement, strain and stress at indices zero, one and two of the elasticity complex, respectively. Mathematically equivalent formulations may lead to different numerical performances. For example, discretizations based on the Hellinger-Reissner principle avoid the locking phenomenon in the computation with the displacement formulation (the fact that the convergence deteriorates as the material becomes incompressible)  \cite{Boffi.D;Brezzi.F;Fortin.M.2013a}. For classical elasticity, the linearized curvature of $\deff u$ vanishes, i.e., $D^{1}\circ D^{0}=0$.  One may also incorporate defects in standard elasticity models as the violation of this compatibility condition \cite{amstutz2019incompatibility}.

It is usually desirable to model a thin elastic body as a two dimensional surface. In such dimension reduction,  one expects that as the thickness of the body goes to zero, the two dimensional model approximates the three dimensional one (a rigorous justification of the convergence can be done in terms of $\Gamma$-convergence, which, however, is not the focus of this paper). Different assumptions in the dimension reduction lead to different models. The Kirchhoff-Love model uses a mid-surface plane to represent the body. With the Kirchhoff-Love assumptions, the plate bending problem boils down to determining the out-plane displacement $w$ (the component of the displacement which is in the transversal direction to the mid-surface of the plate) satisfying a scalar biharmonic equation. That is, the energy functional is $\|\Delta w\|_{A}^{2}-(g, w)$, where   $A$ is a fourth-order tensor. See, for example, \cite[(4)-(5)]{arnold2002range}. In the Reissner-Mindlin model, the displacement is recovered from a scalar field $w$ (out-plane displacement) and a vector field $\theta$. The independent variable $\theta$ represents the rotation angle of the normal vector to the mid-surface. The energy functional for determining $w$ and $\theta$ is $\|\grad w-\theta\|_{E}^{2}+\|\deff \theta\|^{2}_{\tilde C}-(G, \theta)-(F, w)$ \cite[(8)-(10)]{arnold2002range}. Again, $\|\cdot\|_{E}$ and $\|\cdot\|_{\tilde C}$ are weighted $L^{2}$ inner products and $G$ and $F$ are given functions.  In the Kirchhoff-Love model, the assumptions are stronger such that $\theta$ is determined as a first order derivative of $w$, and is thus eliminated from the equations. We refer to \cite{ciarlet1997mathematical} for an in-depth discussion on plate models. 

The Kirchhoff-Love plate, i.e., the biharmonic equation, is naturally related to the Hessian complex, as the quadratic term in the energy functional can be written as a weighted $L^{2}$ inner product of $\hess w$. Nevertheless, cohomological structures have not been identified for the Reissner-Mindlin plate and Cosserat elasticity. As we have seen for linear elasticity, such structures are desirable for finite element exterior calculus and developing new models and analysis. 

The main contribution of this section is to observe that the structures of the Reissner-Mindlin plate and Cosserat elasticity are encoded in the twisted de~Rham complexes. In other words, we establish the horizontal arrows in the diagram \eqref{diagram:models}. The BGG machinery deriving the BGG complex (e.g., standard elasticity) from the twisted de~Rham complex (e.g., Cosserat elasticity) thus has a physical meaning and provides a mathematical approach for eliminating degrees of freedom from physical models. In the rest of this section, we give a precise meaning to such correspondence, and list analogs that come from other complexes.

  }
\subsection{Linear Cosserat model as a Hodge-Laplacian}

We will show that the linear Cosserat model corresponds to a Hodge-Laplacian boundary value problem of the twisted de Rham complex that corresponds to the elasticity complex. To make this statement precise, we will follow the framework of Hilbert complexes \cite{arnold2018finite} and investigate the $L^{2}$ twisted de Rham complex with unbounded operators and its domain complex. The Sobolev regularity required for this purpose is slightly different from the setting in \cite{arnold2021complexes} for deriving the elasticity complex. The difference is similar to the situation discussed in Section \ref{sec:struct}.

In the setting of Sections \ref{sec:struct} and \ref{sec:background}, we take $n=3$, $\frak g=\frak{sl}(n+1,\Bbb R)$ and $\Bbb V=\Lambda^2\mathbb R^{(n+1)*}$. This leads to $N=2$, $\Bbb V_0=\Bbb R^{3*}$ and $\Bbb V_1=\Lambda^2\Bbb R^{3*}$. We denote elements of $Z^{i,0}$ as triples $(\phi_1,\phi_2,\phi_3)$ of $i$-forms and elements of $Z^{i,1}$ as families $\psi_{jk}$ of $i$-forms with two-skew symmetric indices, i.e.\ $\psi_{jj}=0$ and $\psi_{kj}=-\psi_{jk}$. Since there are just two rows, we just have one K-operator, which is given by
  $K(\psi)_j=\sum_\ell x^\ell \psi_{\ell j}$, which leads to $S(\psi)_j=\sum_\ell dx^\ell\wedge\psi_{\ell j}$. In this case, no additional identities need to be verified, so our machinery applies. The form of \cite{arnold2021complexes} is then obtained by passing to vector proxies. This looks as follows. 

Let
$$
d_{V}^{i}:=
\left (
\begin{array}{cc}
d^{i} & -S^{i}\\
0 & d^{i}
\end{array}
\right ),
$$
where $d^{i}=\grad, \curl, \div$ for $i=0, 1, 2$, $S^{0}=-\mskw$, $S^{1} :=\mathcal{S}$ and $S^{2}=2\vskw$. We have $d^{i+1}S^{i}=-S^{i+1}d^{i}$, for $i=0, 1$ and this implies  $d_{V}^{i+1}\circ d_{V}^{i}=0$, $i=0, 1$. 
Below we focus on the following $L^{2}$ twisted de Rham complex with unbounded linear operators
{\footnotesize\begin{equation}\label{L2-twisted-elasticity}
  \begin{tikzcd}
0\arrow{r}&
\left ( \begin{array}{c}
L^{2}\otimes \mathbb{V} \\
L^{2}\otimes \mathbb{V} 
 \end{array}\right )
 \arrow{r}{
d_{V}^{0}
 }&  \left ( \begin{array}{c}
L^{2}\otimes \mathbb{M} \\
L^{2}\otimes \mathbb{M} 
 \end{array}\right ) \arrow{r}{{d_{V}^{1}}} &\left ( \begin{array}{c}
L^{2}\otimes \mathbb{M} \\
L^{2}\otimes \mathbb{M} 
 \end{array}\right )  \arrow{r}{d_{V}^{2}} &   \left ( \begin{array}{c}
L^{2}\otimes \mathbb{V} \\
L^{2}\otimes \mathbb{V} 
 \end{array}\right )\arrow{r}{} &0,
\end{tikzcd}
\end{equation}}
and its domain complex
\begin{equation}\label{domain-twisted-elasticity}
\begin{tikzcd}
0\arrow{r}&
H(d_{V}^{0})
 \arrow{r}{
d_{V}^{0}
 }&  H(d_{V}^{1}) \arrow{r}{{d_{V}^{1}}} &H(d_{V}^{2})  \arrow{r}{d_{V}^{2}} &  H(d_{V}^{3})\arrow{r}{} &0.
\end{tikzcd}
\end{equation}
Here for any linear operator $D$, we define $H(D):=\{u\in L^{2} : Du\in L^{2}\}$.  By defintion, we have $H(d_{V}^{i})=H(d^{i})\oplus H(d^{i})$.  The twisted complex \eqref{domain-twisted-elasticity} is derived from the BGG diagram 
\begin{equation} 
\begin{tikzcd}
0 \arrow{r}{} &H^{1}\otimes \mathbb{V} \arrow{r}{{\grad}} &H(\curl)\otimes \mathbb{V} \arrow{r}{\curl} & H(\div)\otimes \mathbb{V}\arrow{r}{\div} & L^{2}\otimes \mathbb{V}\arrow{r} & 0\\
0 \arrow{r}{}&H^{1}\otimes \mathbb{V} \arrow{r}{{\grad}}\arrow[ur, "-\mskw"] &H(\curl)\otimes \mathbb{V} \arrow{r}{\curl} \arrow[ur, "\mathcal{S}"]&  H(\div)\otimes \mathbb{V}\arrow[ur, "2\vskw"]\arrow{r}{\div}  & L^{2}\otimes \mathbb{V}\arrow{r} & 0
 \end{tikzcd} 
\end{equation}
Note that, as a direct consequence of the identity $dS=-Sd$, we have $S^{i}H(d^{i})\subset H(d^{i+1})$ (for example, $\mathcal{S}H(\curl)\otimes \mathbb{V}\subset H(\div)\otimes \mathbb{V}$).

Using a similar argument as \cite[Theorem 1]{arnold2021complexes}, we conclude that the cohomology of \eqref{domain-twisted-elasticity} is isomorphic to the cohomology of the Sobolev complex 
\begin{equation}\label{q-q-1}
\begin{tikzcd}
0\arrow{r}&
\left ( \begin{array}{c}
H^{q}\otimes \mathbb{V} \\
H^{q}\otimes \mathbb{V} 
 \end{array}\right )
 \arrow{r}{
d_{V}^{0}
 }&  \left ( \begin{array}{c}
H^{q-1}\otimes \mathbb{M} \\
H^{q-1}\otimes \mathbb{M} 
 \end{array}\right ) \arrow{r}{{d_{V}^{1}}} &\left ( \begin{array}{c}
H^{q-2}\otimes \mathbb{M} \\
H^{q-2}\otimes \mathbb{M} 
 \end{array}\right )  \arrow{r}{d_{V}^{2}} &   \left ( \begin{array}{c}
H^{q-3}\otimes \mathbb{V} \\
H^{q-3}\otimes \mathbb{V} 
 \end{array}\right )\arrow{r}{} &0,
\end{tikzcd}
\end{equation}
which further has isomorphic cohomology as the sum of the de Rham complexes for any real number $q$ by applying the argument in Section \ref{sec:twisted}.

We briefly recall the abstract setting of the Hodge-Laplacian boundary value problems \cite{arnold2018finite,Arnold.D;Falk.R;Winther.R.2006a}.  Let $(W^{\bs}, \mathscr{D}^{\bs})$ be a Hilbert complex with unbounded linear operators $\mathscr{D}^{\bs}$, and $(V^{\bs}, \mathscr{D}^{\bs})$ be its domain complex where $\mathscr{D}^{\bs}$ becomes bounded. Let $V^{\ast}_{k}$ be the domain of $\mathscr{D}^{\ast}_{k}$, the adjoint of ${\mathscr{D}^{k-1}}$. The Hodge-Laplacian operator is $L^{k}:=\mathscr{D}^{\ast}_{k+1}\mathscr{D}^{k}+\mathscr{D}^{k-1}\mathscr{D}^{\ast}_{k}$, with the domain $
D(L^{k})=\{u\in V^{k}\cap V_{k}^{\ast} : \mathscr{D}u\in V_{k+1}^{\ast}, \mathscr{D}^{\ast}u\in V^{k-1}\}$.

From this perspective, once we specify a closed Hilbert complex, the Hodge-Laplacian problems and their well-posedness follow from general arguments.

In our case, we choose \eqref{L2-twisted-elasticity} as $(W^{\bs}, \mathscr{D}^{\bs})$ and choose \eqref{domain-twisted-elasticity} as the domain complex. Our main claim is that with suitable inner products, the Hodge-Laplacian problem at index 0 is the linear Cosserat elasticity model. Next, we verify this claim by comparing to the formulation in \cite{cosserat}. In fact, with certain metric $C$ for the space at index 1, we have
\begin{align}\label{cosserat-energy}\nonumber
(d_{V}^{0}(u, &\omega), d_{V}^{0}(u, \omega))_{C}\\\nonumber
&:=\|\grad u +\mskw \omega\|_{C_{1}}^{2}+\|\grad\omega\|_{C_{2}}^{2}\\\nonumber
&:=\mu\|\sym\grad u\|^{2}+\frac{\mu_{c}}{2}\|2\vskw(\grad u+\mskw \omega) \|^{2}+\frac{\lambda}{2}\|\div u\|^{2}\\
&\quad + \frac{\gamma+\beta}{2}\|\sym \grad\omega\|^{2}+\frac{\gamma-\beta}{4}\|\curl \omega\|^{2}+\frac{\alpha}{2}\|\div \omega\|^{2},
\end{align}
where we have used the identities $2\vskw \grad u= \curl u$ and $\vskw\circ\mskw=\mskw^{-1}\circ\skw\circ \mskw=I$.  Here the metrics { $C_{1}$ and $C_{2}$} are defined by the physical parameters. We refer to \cite[Chapter 8]{arnold2018finite} for the use of weighted inner products for the Maxwell equations and linear elasticity.

Now we see that \eqref{cosserat-energy} exactly corresponds to the linear Cosserat model in \cite{cosserat}. In particular, $\mu\|\sym\grad u\|^{2}+\frac{\mu_{c}}{2}\|2\vskw(\grad u+\mskw \omega) \|^{2}+\frac{\lambda}{2}\|\div u\|^{2}$ is the strain energy and $ \frac{\gamma+\beta}{2}\|\sym \grad\omega\|^{2}+\frac{\gamma-\beta}{4}\|\curl \omega\|^{2}+\frac{\alpha}{2}\|\div \omega\|^{2}$ is the curvature energy. Also note that $\vskw$ restricted on skew symmetric matrices corresponds to $\mathrm{axl}$ in \cite{cosserat}, and $\omega$ corresponds to $-\phi$ in \cite{cosserat}. Moreover, $\grad u+\mskw \omega$ is the infinitesimal first Cosserat stretch tensor, which corresponds to the $d_{V}^{0}$ operator in the twisted complex;  $-\grad\omega$ is the micropolar curvature tensor; $\mu, \lambda$ are the classical Lam\'e moduli; $\alpha, \beta, \gamma$ are additional micropolar moduli with the dimension of a force; $\mu_{c}$ is the Cosserat couple modulus.

We also observe that the above discussions for the Hessian complex and its twisted de Rham version correspond to the Kirchhoff model and the modified Reissner-Mindlin plate model, respectively.

The twisted de Rham complex corresponding to the Hessian complex in 2D \cite{arnold2021complexes} is:
\begin{equation}\label{twisted-elasticity-2D}
\begin{tikzcd}
0\arrow{r}&
\left ( \begin{array}{c}
H^{1} \\
 H^{1}\otimes \mathbb{V} 
 \end{array}\right )
 \arrow{r}{
d_V^{0}
 }&  \left ( \begin{array}{c}
 H(\rot)\\
H(\rot)\otimes \mathbb{V}
 \end{array}\right ) \arrow{r}{{d_V^{1}}}  &   \left ( \begin{array}{c}
L^{2}\\
 L^{2}\otimes \mathbb{V}
 \end{array}\right )\arrow{r}{} &0,
\end{tikzcd}
\end{equation}
where 
$$
d_V^{0}=\left ( \begin{array}{cc}
\grad & -I\\
0 & \grad
 \end{array}\right ),\quad
 d_V^{1}=\left ( \begin{array}{cc}
\rot & -\mathrm{sskw}\\
0 & \rot
 \end{array}\right ).
$$
The Hodge-Laplacian problem at index zero has the energy functional 
\begin{equation}\label{eqn:modified-reissner-mindlin}
\|\grad u-\phi\|_{A}^{2}+\|\grad\phi\|_{B}^{2}.
\end{equation}
With proper inner products, this corresponds to a {modified} Reissner-Mindlin plate model \cite[p.1279]{arnold1989uniformly}. The only difference between \eqref{eqn:modified-reissner-mindlin} and the original Reissner-Mindlin model \cite[(10)]{falk2008finite} is that
{ the energy of the latter involves a term of $\deff \phi$, which we denote by $\|\deff\phi\|_{\tilde{B}}^{2}:=\int \tilde{B}^{ijkl}(\deff u)_{ij}(\deff u)_{kl}\, dx$ with a fourth order tensor $\tilde{B}^{ijkl}$.  As $\deff u$ is the symmetric part of $\grad u$, there exists another fourth order tensor $B^{ijkl}$, such that  $\|\grad\phi\|_{{B}}^{2}:=\int {B}^{ijkl}(\grad u)_{ij}(\grad u)_{kl}\, dx=\|\deff u\|_{\tilde{B}}^{2}$. Assume that $\tilde{B}$ is positive definite, i.e., $\|\deff\phi\|_{\tilde{B}}\geq C\|\deff \phi\|$ with a positive constant $C$. Then using the Korn inequality (with proper boundary conditions or orthogonality), we get
$$
\|\grad\phi\|_{{B}}=\|\deff u\|_{\tilde{B}}^{2}\geq C\|\deff \phi\|\geq C' \|\grad\phi\|.
$$
On the other hand, if all the components of $B$ are bounded, we have $\|\grad\phi\|\leq \beta \|\grad\phi\|_{{B}}$ for some constant $\beta$. Therefore $\|\cdot\|_{{B}}$ is non-degenerate and leads to an equivalent $L^{2}$ norm, i.e., there exists positive constants $C_1$ and $C_{2}$ such that
$$
C_{1}\|\grad u\|\leq \|\grad u\|_{B}\leq C_{2}\|\grad u\|.
$$
 }

The Hessian complex in 2D is
\begin{equation}\label{hessian-2D}
\begin{tikzcd}
0\arrow{r}&
H^{2}
 \arrow{r}{
\hess
 }&  H(\rot; \mathbb{S}) \arrow{r}{\rot} &L^{2}\otimes \mathbb{V}\arrow{r}{} & 0.
\end{tikzcd}
\end{equation}
The  Hodge-Laplacian problem of \eqref{hessian-2D} { at index zero} describes the Kirchhoff plate, which is the $\Gamma$-limit of the Reissner-Mindlin plate as the thickness tends to zero. This connection, although not formulated in terms of the twisted complex, was used to derive a discretization for the biharmonic equation \cite{feec-lecture}, \cite[Chapter 5]{quenneville2015new}. 

Our results on finite dimensional uniform representation of cohomology imply the Poincar\'e inequalities for the twisted de Rham complexes (although these inequalities readily follow from a straightforward argument in this particular case), and hence the well-posedness of the variational problem. More importantly, fitting the linear Cosserat model in the cohomological framework makes it possible to bring in tools and general results from the abstract theory. For example, the Cosserat version of the ``rigid body motion'' consists of functions in the kernel of $d_V^{0}$, which have the form { $(a+b\wedge x, b)$}, where $a, b\in \mathbb{R}^{3}$ are constant vectors. This can be obtained by transforming the kernel of $\grad$ in the sum de Rham complex by \eqref{matrix-F}. Once we fix the domain complex, the general theory implies that the  formulations are automatically well-posed and { may also indicate the correct form of the boundary conditions. }
  Moreover, we may obtain different formulations, e.g., Cosserat versions of the primal formulation, the Hellinger-Reissner principle, the intrinsic formulation \cite{ciarlet2009intrinsic}, and the Hu-Washizu principle \cite{washizu1968variational}. More importantly, the idea of modelling defects of linear elasticity by the violation of the Saint-Venant condition (vanishing of the linearized curvature encoded in the elasticity complex) \cite{amstutz2019incompatibility} provides a plausible way to investigate defects of Cosserat models by generalizing the ``Cosserat strain'' $(e, p):=d_{V}^{0}(u, w)$ which satisfies the compatibility $d_{V}^{1}(e, p)=0$ to more general functions that violate this condition. In fact, discussions in this direction can be found in, e.g., \cite[(2.14a, 2.15a)]{gunther1958statik}, where $d_{V}^{1}$ in the twisted de Rham complex is involved. However, a detailed discussion of the above issues is beyond the scope of this article.

\subsection{Potential generalizations}

 For potential generalizations, we first consider diagrams that lead to the conformal deformation complex. In \cite{arnold2021complexes}, the following diagram is used to derive the conformal Hessian complex:
 \begin{equation*}
\begin{tikzcd}
0 \arrow{r} &H^{q}\otimes \mathbb{V}\arrow{r}{\dev\grad}  &H^{q-1}\otimes \mathbb{T} \arrow{r}{\sym\curl}  &H^{q-2}\otimes \mathbb{S}^{1}\arrow{r}{\div\div}  &H^{q-4}\otimes \mathbb{R}\arrow{r}{}&0\\
0 \arrow{r}{}&H^{q-1}\otimes \mathbb{V}\arrow{r}{\deff}\arrow[ur, "-\mskw"]  & H^{q-2}\otimes \mathbb{S}  \arrow{r}{\inc}\arrow[ur, "\mathcal{S}"]  & H^{q-4}\otimes \mathbb{S} \arrow{r}{\div}\arrow[ur, "\tr"]  & H^{q-5}\otimes \mathbb{V} \arrow{r}{} &  0,
\end{tikzcd}
\end{equation*}
where we recall that $S^{1}:=u^{T}-\tr(u)I$.
The motivation was to eliminate the skew symmetric part from $\dev\grad u$ to get the operator in the conformal Korn inequality.

From this diagram, we propose a generalized elasticity model with the energy 
\begin{equation}\label{generalization1}
\alpha\|\dev\grad\phi+\mskw u\|_{C}^{2}+\|\deff u\|_{A}^{2},
\end{equation}
where $\|\cdot \|_{A}$ and $\|\cdot\|_{C}$ are weighted $L^{2}$ norms. In particular, we are interested in the case where $\|\cdot\|_{A}$ is the standard inner product in elasticity, i.e., 
$$
\|\deff u\|_{A}^{2}:=\mu\|\sym\grad u\|^{2}+\frac{\lambda}{2}\|\div u\|^{2}.
$$
Then the formal limit $\alpha\rightarrow 0$ in \eqref{generalization1} corresponds to the standard elasticity.
 
Another possible generalization of the Cosserat model is from the three-row diagram \eqref{conformal-3rows}. In particular, the energy defined by the first operator $d_{V}^{0}$ in the twisted complex is
\begin{equation}\label{generalization2}
\|\grad u -\iota \sigma+\mskw \omega\|_{C_{1}}^{2}+\|{(\grad \sigma -\phi)+(\grad\omega+\mskw \phi)}\|_{C_{2}}^{2}+\|\grad\phi\|_{C_{3}}^{2}.
\end{equation}
Now the first term $\|\grad u -\iota \sigma+\mskw \omega\|_{C_{1}}$ contains an additional $\iota \sigma$ term compared to the Cosserat model. This term can be interpreted as a pointwise dilation, in addition to the rotational degrees of freedom.

From the 2D diagram \eqref{conformal-3rows-2D}, we also obtain a generalization of the plate models, with the energy
\begin{equation}
\|\grad u-\iota \sigma-\mskw \omega\|_{C_{1}}^{2}+\|{(\grad \sigma-\phi)+(\grad\omega-J \phi)}\|_{C_{2}}^{2}+\|\grad\phi\|_{C_{3}}^{2}.
\end{equation}
For each of the models proposed above, analytic results, e.g. Poincar\'e inequalities and well-posedness follow from a standard argument based on complexes \cite[Chapter 2]{arnold2021complexes}. 

Other generalizations of the linear Cosserat model exist, e.g., \cite{neff2009new} with a conformally invariant curvature energy.

\section{Conclusion}
 In this paper, we established some BGG complexes with weak regularity and computed their cohomology. The generality and the geometric and algebraic structures allow us to fix the 2D conformal Korn inequality by adding a third order term in addition to the Cauchy-Riemann operator and achieve generalizations of the linear Cosserat elasticity model from a cohomological perspective. These applications show a deep connection between analysis (e.g. inequalities), geometry, homological algebra and continuum mechanics.

This relaxation of the injectivity/surjectivity conditions in \cite{arnold2021complexes} renders a feasible approach to fit finite element spaces in the BGG diagrams. Therefore  numerical schemes with weakly imposed constraints (symmetry, trace-freeness etc.) with standard de Rham finite elements should be promising (c.f., \cite{arnold2007mixed}).  
We proposed several model problems as generalizations of the Cosserat model from a cohomological perspective. With the idea of the Erlangen program in mind, we leave it as future work to investigate the mechanical significance, numerics and nonlinear versions of the generalized models in Section \ref{sec:Cosserat}.  Applications of these complexes in (numerical) general relativity, e.g., potential-based formulations in the direction of \cite{beig2020linearised}, can also be further directions.

The complexes we have constructed have a close relation to general relativity (in fact, the Cosserat model and general relativity were Cartan's motivation for his development of the concept of torsion \cite{scholz2019cartan}). For example, the conformal deformation complex encodes the York split and the transverse-trace-free (TT) gauge in the description of gravitational waves. Beig and Chrusciel \cite{beig2020linearised} discussed the applications of complexes in the linearized Einstein constraint equations with the conformal deformation complex and the conformal Hessian complexes playing a key role. Our results describe the cohomology (and thus analytic properties) of these complexes. In \cite{beig2020linearised} the authors proceed to define the so-called momentum complex, by a construction that looks similar to repeated application of the  ``two-step construction'' in \cite{arnold2021complexes}. While the momentum complex is not an example of a BGG sequence, we expect that our methods can be applied to the study of this complex. As with other applications to general relativity, this will be taken up elsewhere.

\section*{Appendix: Matrix form of operators}  
 The operators involved in the theorems and proofs in Section \ref{sec:framework} can be given in an explicit component-wise form, which are listed below:
  $$
{G}=-\sum_{k=0}^{\infty}(Td)^{k}T=
-  \left (
  \begin{array}{cccccc}
  0 &0&0&0&\cdots&0\\
  T & 0&0&0&\cdots&0\\
  TdT & T & 0&0&\cdots&0\\
 (Td)^{2}T&  TdT & T & 0 &\cdots&0\\
  &\cdots&\cdots&\cdots&&
  \end{array}
  \right ),
  $$
  $$
  A=I-Gd_{V}=
  \left (
  \begin{array}{cccccc}
  I &0&0&0&\cdots&0\\
  Td & P_{\ker} &0&0&\cdots&0\\
  (Td)^{2} & TdP_{\ker} & P_{\ker}&0&\cdots&0\\
 (Td)^{3}&  (Td)^{2}P_{\ker} & TdP_{\ker} & P_{\ker} &\cdots&0\\
  &\cdots&\cdots&\cdots&&
  \end{array}
  \right ),
  $$
  and
    $$
  d_{V}A=
  \left (
  \begin{array}{cccccc}
  P_{\ran^{\perp}}d &0&0&0&\cdots&0\\
    P_{\ran^{\perp}}dTd &   P_{\ran^{\perp}}dP_{\ker} &0&0&\cdots&0\\
   P_{\ran^{\perp}}(dT)^{2}d &  P_{\ran^{\perp}}dTdP_{\ker} &   P_{\ran^{\perp}}dP_{\ker} &0&\cdots&0\\
 P_{\ran^{\perp}}(dT)^{3}d &P_{\ran^{\perp}}(dT)^{2}dP_{\ker}&  P_{\ran^{\perp}}dTdP_{\ker} &   P_{\ran^{\perp}}dP_{\ker} &\cdots&0\\
  &\cdots&\cdots&\cdots&&
  \end{array}
  \right ).
  $$
  Restricted to $\Upsilon$, $D= {P_{\ker}}d_{V}A$ becomes  
     \begin{equation}\label{D-matrix}
  D=P_{\Upsilon}d_{V}A=
  \left (
  \begin{array}{cccccc}
P_{\Upsilon}d &0&0&0&\cdots&0\\
P_{\Upsilon}dTd &   P_{\Upsilon}d  &0&0&\cdots&0\\
P_{\Upsilon}(dT)^{2}d &  P_{\Upsilon}dTd  &  P_{\Upsilon}d &0&\cdots&0\\
P_{\Upsilon}(dT)^{3}d &P_{\Upsilon}(dT)^{2}d  & P_{\Upsilon}dTd &   P_{\Upsilon}d  &\cdots&0\\
&\cdots&\cdots&\cdots&&
  \end{array}
  \right ).
  \end{equation}
  Furthermore, we have
  $$
  B=P_{\Upsilon}(I-d_{V}G)=\left (
  \begin{array}{cccccc}
  P_{\Upsilon} & 0 & 0 & 0 & \cdots & 0\\
    P_{\Upsilon}dT & P_{\Upsilon} & 0 & 0 & \cdots & 0\\
        P_{\Upsilon}(dT)^{2} & P_{\Upsilon}dT &P_{\Upsilon}  & 0 & \cdots & 0\\
                P_{\Upsilon}(dT)^{3} &   P_{\Upsilon}(dT)^{2}  &P_{\Upsilon}dT  & P_{\Upsilon} & \cdots & 0\\
  &\cdots&\cdots&\cdots&&
  \end{array}
  \right ).
  $$
  From Theorem \ref{thm:input-twisted} and Theorem \ref{thm:twisted-bgg}, the composition $B\circ F$ gives a cohomology-preserving map from the input complexes to the BGG complex. The component-wise form of $B\circ F$ is the following:
  
 \resizebox{\textwidth}{!}{$%
   $$
  B\circ F=\left (
  \begin{array}{ccccc}
  P_{\Upsilon} & P_{\Upsilon}K & \frac{1}{2}P_{\Upsilon}K^{2} & \frac{1}{6}P_{\Upsilon}K^{3} & \cdots \\
    P_{\Upsilon}dT & P_{\Upsilon}dTK+P_{\Upsilon} & \frac{1}{2}P_{\Upsilon}dTK^{2}+P_{\Upsilon}K & \frac{1}{6}P_{\Upsilon}dTK^{3}+\frac{1}{2}P_{\Upsilon}K^{2}  & \cdots \\
        P_{\Upsilon}(dT)^{2} & P_{\Upsilon}(dT)^{2}K+P_{\Upsilon}dT &\frac{1}{2}P_{\Upsilon}(dT)^{2}K^{2}+ P_{\Upsilon}dTK+P_{\Upsilon}  & \frac{1}{6}P_{\Upsilon}(dT)^{2}K^{3}+ \frac{1}{2}P_{\Upsilon}dTK^{2}+ P_{\Upsilon}K & \cdots \\
        &\cdots&\cdots&\cdots&
  \end{array}
  \right ).
  $$ 
    $}%

\section*{Acknowledgments}
The first author is supported by the Austrian Science Fund (FWF): P33559-N. The second author was supported   by a Hooke Research Fellowship and a Royal Society University Research Fellowship.

The authors would like to thank Douglas N. Arnold, Andrea Dziubek and Endre Suli for helpful discussions.

\bibliographystyle{siam}      
\bibliography{reference}{}   

\end{document}